\documentclass[11pt]{article}

\usepackage{amssymb}
\usepackage{amsmath}
\usepackage{amsthm}
\usepackage{tikz-cd}

\numberwithin{equation}{section}

\newtheoremstyle{slplain}
  {\topsep}
  {\topsep}
  {\slshape}
  {0pt}
  {\bfseries}
  {.}
  {0.5em}
  {}

\theoremstyle{slplain}
  \newtheorem{THM}{Theorem}[section]
  \newtheorem{LEM}[THM]{Lemma}
  \newtheorem{PROP}[THM]{Proposition}
  \newtheorem{COR}[THM]{Corollary}

\theoremstyle{definition}
  
  \newtheorem{EX}{Example}[section]

\newcommand\scat[1]{\langle#1\rangle}

\renewcommand{\le}{\leqslant}
\renewcommand{\ge}{\geqslant}

\newcommand{\0}{\varnothing}

\renewcommand{\sec}{\cap}
\renewcommand{\phi}{\varphi}
\renewcommand{\epsilon}{\varepsilon}
\renewcommand{\AA}{\mathbf{A}}
\newcommand{\BB}{\mathbf{B}}
\newcommand{\CC}{\mathbf{C}}
\newcommand{\Ch}{\mathbf{Ch}}
\newcommand{\DD}{\mathbf{D}}

\newcommand{\KK}{\mathbf{K}}

\newcommand{\NN}{\mathbb{N}}

\newcommand{\QQ}{\mathbb{Q}}

\renewcommand{\SS}{\mathbf{S}}

\newcommand{\union}{\cup}
\newcommand{\restr}[2]{\hbox{$#1$}\hbox{$\upharpoonright$}_{#2}}

\newcommand{\Boxed}[1]{\mbox{$#1$}}
\newcommand{\id}{\mathrm{id}}

\newcommand{\Ob}{\mathrm{Ob}}

\newcommand{\Sub}{\mathbf{Sub}}
\newcommand{\dom}{\mathrm{dom}}
\newcommand{\cod}{\mathrm{cod}}
\newcommand{\mor}{\mathit{mor}}
\newcommand{\toCC}[1]{\overset{#1}{\longrightarrow}}
\newcommand{\fin}{\mathit{fin}}
\newcommand{\Gra}{\mathbf{Gra}}

\newcommand{\calA}{\mathcal{A}}
\newcommand{\calB}{\mathcal{B}}
\newcommand{\calC}{\mathcal{C}}

\newcommand{\calP}{\mathcal{P}}

\newcommand{\calR}{\mathcal{R}}
\newcommand{\calS}{\mathcal{S}}
\newcommand{\calT}{\mathcal{T}}
\newcommand{\calU}{\mathcal{U}}
\newcommand{\calV}{\mathcal{V}}

\newcommand{\iso}{\mathrm{iso}}

\newcommand{\Aut}{\mathrm{Aut}}

\title{Ramsey degrees: big v.\ small\footnote{%
  This revised version of the paper differs from the published one in the formulation of
  Theorem \ref{rdbas.thm.minT}; specifically, it now includes the additional assumption that 
  $S$ is locally finite. My thanks go to Maximilian Strohmeier for pointing out the omission.}}
\author{%
  Dragan Ma\v sulovi\'c\\
  University of Novi Sad, Faculty of Sciences\\
  Department of Mathematics and Informatics\\
  Trg Dositeja Obradovi\'ca 3, 21000 Novi Sad, Serbia\\
  e-mail: dragan.masulovic@dmi.uns.ac.rs}

\begin{document}
\maketitle

\begin{abstract}
  In this paper we investigate algebraic properties of big Ramsey degrees
  in categories satisfying some mild conditions.
  As the first nontrivial consequence of the generalization we advocate in this paper
  we prove that small Ramsey degrees are the minima of the corresponding big ones.
  We also prove that big Ramsey degrees are subadditive and show that equality is enforced by an abstract property of objects
  we refer to as self-similarity. Finally,
  we apply the abstract machinery developed in the paper to show that if a countable relational
  structure has finite big Ramsey degrees, then so do its quantifier-free reducts. In particular, it follows that
  the reducts of $(\QQ, \Boxed<)$, the random graph, the random tournament and $(\QQ, \Boxed<, 0)$
  all have finite big Ramsey degrees.

  \bigskip


  \noindent \textbf{AMS Subj.\ Classification (2010):} 05C55, 18A99
\end{abstract}

\section{Introduction}

Generalizing the finite version of Ramsey's Theorem \cite{Ramsey}, the structural Ramsey theory originated at
the beginning of 1970’s in a series of papers (see \cite{N1995} for references).
We say that a class $\KK$ of finite structures has the \emph{Ramsey property} if the following holds:
for any number $k \ge 2$ of colors and all $\calA, \calB \in \KK$ there is a $\calC \in \KK$ such that
$$
  \calC \longrightarrow (\calB)^\calA_k.
$$
The above is a symbolic way of expressing that
no matter how we color the copies of $\calA$ in $\calC$ with $k$ colors, one can always find a \emph{monochromatic} copy
$\calB'$ of $\calB$ in $\calC$ (that is, all the copies of $\calA$ that fall within $\calB'$ are colored by the same color).

Many natural classes of structures such as finite graphs and finite posets
do not have the Ramsey property. Nevertheless, many of these classes enjoy the weaker property of \emph{having
finite (small) Ramsey degrees} first observed in~\cite{fouche97}.
An integer $t \ge 1$ is a \emph{(small) Ramsey degree of a structure} $\calA \in \KK$ if it is the smallest
positive integer satisfying the following: for any $k \ge 2$ and
any $\calB \in \KK$ there is a $\calC \in \KK$ such that
$$
  \calC \longrightarrow (\calB)^\calA_{k, t}.
$$
This is a symbolic way of expressing that 
no matter how we color the copies of $\calA$ in $\calC$ with $k$ colors, one can always find a \emph{$t$-oligochromatic} copy
$\calB'$ of $\calB$ in $\calC$ (that is, there are at most $t$ colors used to color the copies of $\calA$ that fall within $\calB'$).
If no such $t \ge 1$ exists for an $\calA \in \KK$, we say that $\calA$ \emph{does not have finite (small) Ramsey degree.}
For example, finite graphs, finite posets and many other classes of finite structures
are known to have finite (small) Ramsey degrees \cite{fouche97,fouche98,fouche99}.

The infinite version of Ramsey's Theorem \cite{Ramsey} claims that given a finite chain $n$,
no matter how we color the copies of $n$ in the chain $\omega = \{0, 1, 2, \ldots\}$ with $k$ colors, one can always find a monochromatic
copy of~$\omega$ inside~$\omega$. Interestingly, the same is not true for $\QQ$. One can
easily produce a Sierpi\'nski-style coloring of two-element subchains of $\QQ$ with two colors
and with no monochromatic copy of~$\QQ$.
However, for every coloring of two-element subchains of $\QQ$ with $k$ colors
one can always find a 2-oligochromatic copy of $\QQ$
\cite{galvin1,galvin2}. This result was then generalized in \cite{devlin}
where for each $m$ a positive integer $T_m$ was computed so that
for every coloring of $m$-element subchains of $\QQ$ one can always find a $T_m$-oligochromatic copy of $\QQ$.
The integer $T_m$ is referred to as the \emph{big Ramsey degree of~$m$ in~$\QQ$}.

Following \cite{KPT} where the study of this general notion was explicitly suggested for the first time,
an integer $T \ge 1$ is a \emph{big Ramsey degree of a finite substructure $\calA$ of a countably infinite
structure $\calU$}  if it is the smallest positive integer such that
for every coloring $\chi : \binom \calU \calA \to k$ one can always find a $T$-oligochromatic copy of $\calU$ inside $\calU$.
(Here, $\binom \calU \calA$ denotes the set of all the substructures of $\calU$ isomorphic to $\calA$.)
If no such $T$ exists, we say that \emph{$\calA$ does not have big Ramsey degree in $\calU$}.
We denote the big Ramsey degree of $\calA$ in $\calU$ by $T(\calA, \calU)$, and write
$T(\calA, \calU) = \infty$ if $\calA$ does not have the big Ramsey degree in~$\calU$.
We say that a countably infinite structure $\calU$ \emph{has finite big Ramsey degrees} if
$T(\calA, \calU) < \infty$ for every finite substructure $\calA$ of~$\calU$.

As the structural Ramsey theory evolved, it has become evident that the Ramsey property for a class of objects
depends not only on the choice of objects, but also on the choice of morphisms involved
(see \cite{GLR,Nesetril,mu-pon,vanthe-ufcs,P-V,Zucker-1}). This is why we believe that
category theory is a convenient ambient to consider Ramsey-theoretic notions.
It was Leeb who pointed out already in 1970 \cite{leeb-cat} that the use of category theory can be quite helpful
both in the formulation and in the proofs of results pertaining to structural Ramsey theory.
However, instead of pursuing the original approach by Leeb (which has very fruitfully been applied to a wide range of
Ramsey-type problems \cite{GLR, leeb-cat, Nesetril-Rodl}),
we proposed in~\cite{masulovic-ramsey} a systematic study of
a simpler approach motivated by and implicit in \cite{mu-pon,vanthe-more,Zucker-1}.
We have shown in~\cite{masulovic-ramsey} that the Ramsey property is a genuine categorical property by proving that it
is preserved by categorical equivalence.

Another observation that crystallized over the years is the fact that we can and have to distinguish
between the Ramsey property for structures (where we color \emph{copies} of one structure within another structure)
and the Ramsey property for embeddings (where we color \emph{embeddings} of one structure into another structure).
In the categorical reinterpretation of these notions
we shall, therefore, consider the Ramsey property for objects and the Ramsey property for morphisms.
Consequently, we shall have to introduce small and big Ramsey degrees for both objects and morphisms.
Although Ramsey degrees for objects are true generalizations of Ramsey degrees for structures, it turns out that
Ramsey degrees for morphisms are easier to calculate with. Fortunately, the relationship between the two is straightforward,
as demonstrated in \cite{Zucker-1,Zucker-2}, and it carries over to the abstract case of Ramsey degrees in
categories (see Propositions~\ref{rdbas.prop.sml} and~\ref{rdbas.prop.big}).
In this paper we put together and generalize several ideas from \cite{dasilvabarbosa,vanthe-more,Zucker-1,Zucker-2} to
obtain several purely categorical results. We then use this more abstract setting to offer new insights into the
relationship between the small and big Ramsey degrees.

In Section~\ref{rdbas.sec.prelim} we give a brief overview of standard notions of
category theory and in particular reflect on the observation from \cite{dasilvabarbosa} that
expansions of classes of structures as introduced in \cite{KPT,vanthe-more} are nothing but special
forgetful functors.

In Section~\ref{rdbas.sec.ramsey-degs} we present a
reinterpretation of the various notions of structural Ramsey theory in the language of category theory.

It was proved in \cite{Zucker-2} in the context of
relational structures that small Ramsey degrees are not larger than the corresponding big ones.
As the first nontrivial benefit of the generalization we advocate in this paper
we prove in Section~\ref{rdbas.sec.min} that more is true. It turns out that
small Ramsey degrees are the minima of the corresponding big ones in the following sense:
for every category $\DD$ satisfying certain mild conditions and every object $A$ in that category we have that
$$
  t_\DD(A) = \min_{\SS, S} T_\SS(A, S),
$$
where the minimum is taken over all the categories $\SS$ that contain $\DD$ as a full subcategory,
and all the objects $S$ of $\SS$ which are universal for $\DD$. (The nonstandard notions will be specified below, of course;
in particular, $t_\DD(A)$ is defined at the beginning of Section~\ref{rdbas.sec.ramsey-degs}, while
$T_\SS(A, S)$ is defined immediately after Proposition~\ref{rdbas.prop.sml}.)
The intuition behind the construction the proof relies on is that computing the small Ramsey degree of an
object $A$ within a category $\DD$ is analogous to computing the big Ramsey degree of the same object $A$
in the category $\DD$ considered as an object of a larger category (which contains both $A$ and $\DD$ as its objects).

In Section~\ref{rdbas.sec.mon} we generalize several facts about the monotonicity of Ramsey degrees which
were first observed in \cite{Zucker-1,Zucker-2} in the context of relational structures.
We show that in some cases the big Ramsey degrees are
monotonous in the first argument. The fact that small Ramsey degrees are the minima of the corresponding
big Ramsey degrees immediately yields the monotonicity of the small Ramsey degrees (which was proved directly in \cite{Zucker-2} for relational structures).
This result is intriguing because we end up with a proof of a property of small Ramsey degrees that
follows from the analogous property of the big Ramsey degrees.

In Section~\ref{rdbas.sec.additivity-big} we generalize a result
from~\cite{dasilvabarbosa} about the additivity of big Ramsey degrees.
Given an expansion $U : \CC^* \to \CC$ satisfying certain mild conditions we prove that
$$
  T_\CC(A, S) \le \sum T_{\CC^*}(A^*, S^*),
$$
where the sum is taken over all the expansions $A^*$ of $A$.
We then identify an abstract property of objects we refer to as \emph{self-similarity}
and prove that the equality holds in the above identity involving the big Ramsey degrees whenever $S^*$ is self-similar.

In Section~\ref{rdbas.sec.appls} we apply the abstract machinery developed in the paper to show that
if a countably infinite relational structure has finite big Ramsey degrees, then so do its quantifier-free reducts.
In particular, reducts of a many combinatorially interesting structures such as
$(\QQ, \Boxed<)$, the random graph, the random poset and the random tournament
all have finite big Ramsey degrees.
Moreover, we prove that if an ultrahomogeneous countably infinite structure has finite big Ramsey degrees, then so
does the structure obtained from it by adding finitely many constants. It follows that all the 116 reducts of
$(\QQ, \Boxed<, 0)$ also have finite big Ramsey degrees.

\section{Preliminaries}
\label{rdbas.sec.prelim}

In this section we provide a brief overview of elementary category-theoretic notions.
For a detailed account of category theory we refer the reader to~\cite{AHS}.

In order to specify a \emph{category} $\CC$ one has to specify
a class of objects $\Ob(\CC)$, a class of morphisms $\hom_\CC(A, B)$ for all $A, B \in \Ob(\CC)$,
the identity morphism $\id_A$ for all $A \in \Ob(\CC)$, and
the composition of morphisms~$\cdot$~so that
$\id_B \cdot f = f = f \cdot \id_A$ for all $f \in \hom_\CC(A, B)$, and
$(f \cdot g) \cdot h = f \cdot (g \cdot h)$ whenever the compositions are defined.
We write $A \toCC{\CC} B$ as a shorthand for $\hom_\CC(A, B) \ne \0$. If $f \in \hom_\CC(A, B)$ then we write
$\dom(f) = A$ and $\cod(f) = B$.

A category $\CC$ is \emph{locally small} if $\hom_\CC(A, B)$ is a set for all $A, B \in \Ob(\CC)$.
Sets of the form $\hom_\CC(A, B)$ are then referred to as \emph{hom-sets}.
A locally small category $\CC$ is \emph{small} if $\Ob(\CC)$ is a set. Hence, in a small category
we have a set of objects and all the hom-sets are indeed sets.

Most results in this paper apply to locally small categories. However, in order to ensure that
the outcome of the power construction (see Section~\ref{rdbas.sec.min}) has hom-\emph{sets},
in the second part of the paper we actually have work with small categories.
Downscaling to small categories does not affect the applicability of our results
since the objects of our study are combinatorial in nature and invariant under
isomorphisms. Our primary interest is in categories of finite or countably infinite
first-order structures, and to understand the behavior of small and big Ramsey degrees in this
context it suffices to consider a single representative of each isomorphism class.

A morphism $f \in \hom_\CC(B, C)$ is \emph{mono} or \emph{left cancellable} if
$f \cdot g = f \cdot h$ implies $g = h$ for all $g, h \in \hom_\CC(A, B)$ where $A \in \Ob(\CC)$ is arbitrary.
A morphism $f \in \hom_\CC(B, C)$ is \emph{invertible} if there is a morphism $g \in \hom_\CC(C, B)$
such that $g \cdot f = \id_B$ and $f \cdot g = \id_C$.
By $\iso_\CC(A, B)$ we denote the set of all invertible
morphisms $A \to B$, and we write $A \cong B$ if $\iso_\CC(A, B) \ne \0$. Let $\Aut(A) = \iso(A, A)$.
An object $A \in \Ob(\CC)$ is \emph{rigid} if $\Aut(A) = \{\id_A\}$.
A category $\CC$ is \emph{directed} if for all $A, B \in \Ob(\CC)$ there is a $C \in \Ob(\CC)$ such that
$A \toCC{\CC} C$ and $B \toCC{\CC} C$.
A category $\CC$ \emph{has amalgamation} if for all $A, B, C \in \Ob(\CC)$ and all
$f_1 \in \hom_\CC(A, B)$ and $g_1 \in \hom_\CC(A, C)$ there is a $D \in \Ob(\CC)$ and
morphisms $f_2 \in \hom_\CC(B, D)$ and $g_2 \in \hom_\CC(C, D)$ such that the following diagram commutes:
\begin{center}
  \begin{tikzcd}
    C \arrow[r, "g_2"]   & D\\
    A \arrow[r, "f_1"] \arrow[u, "g_1"] & B \arrow[u, "f_2"']
  \end{tikzcd}
\end{center}

A category $\DD$ is a \emph{subcategory} of a category $\CC$ if $\Ob(\DD) \subseteq \Ob(\CC)$ and
$\hom_\DD(A, B) \subseteq \hom_\CC(A, B)$ for all $A, B \in \Ob(\DD)$.
A category $\DD$ is a \emph{full subcategory} of a category $\CC$ if $\Ob(\DD) \subseteq \Ob(\CC)$ and
$\hom_\DD(A, B) = \hom_\CC(A, B)$ for all $A, B \in \Ob(\DD)$.
We say that a full subcategory $\DD$ of $\CC$ is \emph{cofinal in $\CC$} if for every
$C \in \Ob(\CC)$ there is a $D \in \Ob(\DD)$ with $C \toCC{\CC} D$.

Let $\DD$ be a full subcategory of $\CC$.
An $S \in \Ob(\CC)$ is \emph{universal for $\DD$} if for every $D \in \Ob(\DD)$
the set $\hom_\CC(D, S)$ is nonempty and consists of monos only.
Note that if there exists an $S \in \Ob(\CC)$ universal for $\DD$ then all the morphisms
in $\DD$ are mono. (To see this, take $A, B, C \in \Ob(\DD)$, any $f \in \hom_\DD(A, B)$ and
$g, h \in \hom_\DD(C, A)$ such that $f \cdot g = f \cdot h$. Let $S \in \Ob(\CC)$ be universal for $\DD$.
Then there is some $b \in \hom_\CC(B, S)$ which is mono by the fact that $S$ is universal for $\DD$.
Let $a = b \cdot f \in \hom_\CC(A, S)$. Note that $a$ is also mono. Now, $f \cdot g = f \cdot h$ implies that
$b \cdot f \cdot g = b \cdot f \cdot h$, that is, $a \cdot g = a \cdot h$. But $a$ is mono, so $g = h$.)

For categories $\CC$ and $\DD$, the objects of the \emph{product category} $\CC \times \DD$ are all the pairs
$(C, D)$ where $C$ is an object of $\CC$ and $D$ is an object of $\DD$. The morphisms in $\CC \times \DD$
are all the pairs $(f, g)$ where $f$ is a morphism in $\CC$ and $g$ is a morphism in $\DD$ and
$\id_{(C, D)} = (\id_C, \id_D)$, $\dom(f, g) = (\dom(f), \dom(g))$,
$\cod(f, g) = (\cod(f), \cod(g))$ and $(f_1, g_1) \cdot (f_2, g_2) = (f_1 \cdot f_2, g_1 \cdot g_2)$
whenever the compositions are defined.

A \emph{functor} $F : \CC \to \DD$ from a category $\CC$ to a category $\DD$ maps $\Ob(\CC)$ to
$\Ob(\DD)$ and maps morphisms of $\CC$ to morphisms of $\DD$ so that
$F(f) \in \hom_\DD(F(A), F(B))$ whenever $f \in \hom_\CC(A, B)$, $F(f \cdot g) = F(f) \cdot F(g)$ whenever
$f \cdot g$ is defined, and $F(\id_A) = \id_{F(A)}$.

A functor $U : \CC \to \DD$ is \emph{forgetful} if it is injective on morphisms in the following sense:
for all $A, B \in \Ob(\CC)$ and all $f, g \in \hom_\CC(A, B)$, if $f \ne g$ then $U(f) \ne U(g)$.
In this setting we may actually assume that $\hom_{\CC}(A, B) \subseteq \hom_\DD(U(A), U(B))$ for all $A, B \in \Ob(\CC)$.
The intuition behind this point of view is that $\CC$ is a category of structures, $\DD$ is the category of sets
and $U$ takes a structure $\calA$ to its underlying set $A$ (thus ``forgetting'' the structure). Then for every
morphism $f : \calA \to \calB$ in $\CC$ the same map is a morphism $f : A \to B$ in $\DD$.
Therefore, we shall always take that $U(f) = f$ for all the morphisms in $\CC$. In particular,
$U(\id_{A}) = \id_{U(A)}$ and we, therefore, identify $\id_{A}$ with $\id_{U(A)}$.
Also, if $U : \CC \to \DD$ is a forgetful functor and all the morphisms in $\DD$ are mono, then
all the morphisms in $\CC$ are mono.

Following the model-theoretic notation, a forgetful functor $U : \CC^* \to \CC$
which is surjective on objects will be referred to as \emph{expansion} (cf.\ \cite{dasilvabarbosa,masul-dual-kpt}).
We shall also say that $\CC^*$ is an \emph{expansion} of $\CC$ if $U$ is obvious from the context.
Clearly, if $U : \CC^* \to \CC$ is an expansion, all the morphisms in $\CC$ are mono and $S^* \in \Ob(\CC^*)$ is universal for $\CC^*$
then $U(S^*) \in \Ob(\CC)$ is universal for $\CC$. For $A \in \Ob(\CC)$ let
$$
  U^{-1}(A) = \{A^* \in \Ob(\CC^*) : U(A^*) = A\}.
$$

An expansion $U : \CC^* \to \CC$ is \emph{reasonable} (cf.\ \cite{KPT,masul-dual-kpt}) if
for all $A, B \in \Ob(\CC)$, all $f \in \hom_\CC(A, B)$ and all $A^* \in \Ob(\CC^*)$ with $U(A^*) = A$
there is a $B^* \in \Ob(\CC^*)$ such that $U(B^*) = B$ and $f \in \hom_{\CC^*}(A^*, B^*)$:

\begin{center}
  \begin{tikzcd}
    A^* \arrow[r, "f"] \arrow[d, mapsto, dashed, "U"'] & B^* \arrow[d, mapsto, dashed, "U"] \\
    A \arrow[r, "f"] & B
  \end{tikzcd}
\end{center}

An expansion $U : \CC^* \to \CC$ \emph{has restrictions} \cite{masul-dual-kpt} if
for all $A, B \in \Ob(\CC)$, all $f \in \hom_\CC(A, B)$ and all $B^* \in \Ob(\CC^*)$
with $U(B^*) = B$ there is an $A^* \in \Ob(\CC^*)$ such that $U(A^*) = A$ and
$f \in \hom_{\CC^*}(A^*, B^*)$.
\begin{center}
\begin{tikzcd}
    A^* \arrow[r, "f"] \arrow[d, mapsto, dashed, "U"'] & B^* \arrow[d, mapsto, dashed, "U"] \\
    A \arrow[r, "f"] & B
\end{tikzcd}
\end{center}
If such an $A^*$ is always unique we say that $U : \CC^* \to \CC$ \emph{has unique restrictions}.
We then write $A^* = \restr{B^*}{f}$.

The proofs of the following three lemmas are straightforward:

\begin{LEM}\label{sbrd.lem.1}
  Let $\CC$ and $\CC^*$ be locally small categories and
  let $U : \CC^* \to \CC$ be an expansion with unique restrictions.
  
  $(a)$ Let $A \in \Ob(\CC)$ and $A^*, A^*_1 \in U^{-1}(A)$. Let $f : A^*_1 \to A^*$ be a morphism.
  If $U(f) = \id_A$ then $A^* = A^*_1$ and $f = \id_{A^*}$.
  
  $(b)$ Let $A, B \in \Ob(\CC)$ and let
  $f : A \to B$ be an isomorphism in $\CC$. Take any $B^* \in U^{-1}(B)$ and let $A^* = \restr{B^*}{f}$. Then
  $f : A^* \to B^*$ is an isomorphism in $\CC^*$.
\end{LEM}

\begin{LEM}\label{sbrd.lem.disj-union}
  Let $\CC$ and $\CC^*$ be locally small categories.
  
  $(a)$ The expansion $U : \CC^* \to \CC$ is an expansion with restrictions if and only if
  for all $A \in \Ob(\CC)$ and all $B^* \in \Ob(\CC^*)$ we have that
  $\hom_\CC(A, U(B^*)) = \bigcup_{A^* \in U^{-1}(A)} \hom_{\CC^*}(A^*, B^*)$.

  $(b)$ The expansion $U : \CC^* \to \CC$ is an expansion with unique restrictions if and only if
  for all $A \in \Ob(\CC)$ and all $B^* \in \Ob(\CC^*)$ we have that
  $\hom_\CC(A, U(B^*)) = \bigcup_{A^* \in U^{-1}(A)} \hom_{\CC^*}(A^*, B^*)$
  and this is a disjoint union.
\end{LEM}

\begin{LEM}\label{sbrd.lem.iso-disj-union}
  Let $\CC$ and $\CC^*$ be locally small categories and
  let $U : \CC^* \to \CC$ be an expansion with unique restrictions.
  For $A \in \Ob(\CC)$ let $A^* \in U^{-1}(A)$ be arbitrary, and let
  $I$ be the class of all those $A^{**} \in \Ob(\CC^*)$ such that $A^{**}$ is isomorphic to $A^*$ and
  $U(A^{**}) = A$.

  $(a)$
  $\Aut_\CC(A) = \bigcup_{A^{**} \in I} \iso_{\CC^*}(A^{**}, A^*)$
  and this is a disjoint union. Moreover, $I$ is a set.
  
  $(b)$ Suppose that $I$ is finite and that $\Aut(A)$ is finite.
  Then $|\Aut_\CC(A)| = |I| \cdot |\Aut_{\CC^*}(A^*)|$.
\end{LEM}

An expansion $U : \CC^* \to \CC$ \emph{has the expansion property} (cf.\ \cite{vanthe-more}) if for every $A \in \Ob(\CC)$
there exists a $B \in \Ob(\CC)$ such that $A^* \toCC{\CC^*} B^*$ whenever
$U(A^*) = A$ and all $U(B^*) = B$.

\section{Ramsey degrees in a category}
\label{rdbas.sec.ramsey-degs}

For a $k \in \NN$, a $k$-coloring of a set $S$ is any mapping $\chi : S \to k$, where,
as usual, we identify $k$ with $\{0, 1,\ldots, k-1\}$.

Let $\CC$ be a locally small category and $A, B \in \Ob(\CC)$.
Define $\sim_A$ on $\hom(A, B)$ as follows: for $f, f' \in \hom(A, B)$ we let $f \sim_A f'$ if $f' = f \cdot \alpha$
for some $\alpha \in \Aut(A)$. Then
$$
  \binom B A = \hom(A, B) / \Boxed{\sim_A}
$$
corresponds to all subobjects of $B$ isomorphic to $A$.
For an integer $k \ge 2$ and $A, B, C \in \Ob(\CC)$ we write
$$
  C \longrightarrow (B)^{A}_{k, t}
$$
to denote that for every $k$-coloring
$
  \chi : \binom CA \to k
$
there is a morphism $w : B \to C$ such that $|\chi(w \cdot \binom BA)| \le t$.
(Note that $w \cdot (f / \Boxed{\sim_A}) = (w \cdot f) / \Boxed{\sim_A}$ for $f / \Boxed{\sim_A} \in \binom B A$.)
Instead of $C \longrightarrow (B)^{A}_{k, 1}$ we simply write
$C \longrightarrow (B)^{A}_{k}$.
Analogously, we write
$$
  C \overset{\mor}\longrightarrow (B)^{A}_{k,t}
$$
to denote that for every $k$-coloring
$
  \chi : \hom(A, C) \to k
$
there is a morphism $w : B \to C$ such that $|\chi(w \cdot \hom(A, B))| \le t$.
Instead of $C \overset{\mor}\longrightarrow (B)^{A}_{k, 1}$ we simply write
$C \overset{\mor}\longrightarrow (B)^{A}_{k}$.

A locally small category $\CC$ has the \emph{Ramsey property for objects} if
for every integer $k \ge 2$ and all $A, B \in \Ob(\CC)$ there is a
$C \in \Ob(\CC)$ such that $C \longrightarrow (B)^{A}_k$.
Analogously, $\CC$ has the \emph{Ramsey property for morphisms} if
for every integer $k \ge 2$ and all $A, B \in \Ob(\CC)$ there is a
$C \in \Ob(\CC)$ such that $C \overset{\mor}\longrightarrow (B)^{A}_k$.

Let $\NN = \{1, 2, 3, \ldots \}$ be the set of positive integers and let $\NN_\infty = \NN \union \{\infty\}$.
The usual linear order on the positive integers extends to $\NN_\infty$ straightforwardly:
$$
  1 < 2 < \ldots < \infty.
$$
Ramsey degrees, both big and small, will take their values in $\NN_\infty$, so when we write
$t_1 \ge t_2$ for some Ramsey degrees $t_1$ and $t_2$ then
  $t_1, t_2 \in \NN$ and $t_1 \ge t_2$; or
  $t_1 = \infty$ and $t_2 \in \NN$; or
  $t_1 = t_2 = \infty$.
For notational convenience, if $A$ is an infinite set we shall simply write
$|A| = \infty$ regardless of the actual cardinal~$|A|$. Hence, if $t$ is a Ramsey degree
and $A$ is a set, by $t \ge |A|$ we mean the following:
  $t \in \NN$, $|A| \in \NN$ and $t \ge |A|$; or
  $t = \infty$ and $|A| \in \NN$; or
  $A$ is an infinite set and $t = \infty$.
On the other hand, if $A$ and $B$ are sets then $|A| \ge |B|$ has the usual meaning.

Let $\CC$ be a locally small category.
For $A \in \Ob(\CC)$ let $t_\CC(A)$ denote the least positive integer $n$ such that
for all $k \ge 2$ and all $B \in \Ob(\CC)$ there exists a $C \in \Ob(\CC)$ such that
$C \longrightarrow (B)^{A}_{k, n}$, if such an integer exists.
Otherwise put $t_\CC(A) = \infty$.
Analogously let $t^{\mor}_\CC(A)$ denote the least positive integer $n$ such that
for all $k \ge 2$ and all $B \in \Ob(\CC)$ there exists a $C \in \Ob(\CC)$ such that
$C \overset{\mor}\longrightarrow (B)^{A}_{k, n}$, if such an integer exists.
Otherwise put $t^{\mor}_\CC(A) = \infty$.

The following result was proved for relational structures in \cite{Zucker-1} and generalized to this form in \cite{masul-dual-kpt}:

\begin{PROP}\label{rdbas.prop.sml} (\cite{masul-dual-kpt})
  Let $\CC$ be a locally small category such that all the morphisms in $\CC$ are mono
  and let $A \in \Ob(\CC)$. Then $t^{\mor}_\CC(A)$ is finite if and only if both $t_\CC(A)$ and $\Aut(A)$ are finite,
  and in that case
  $$
    t^{\mor}_\CC(A) = |\Aut(A)| \cdot t_\CC(A).
  $$
\end{PROP}

Note, also, that the above relationship between $t$ and $t^{\mor}$ provides the link between the
Ramsey property for objects and the Ramsey property for morphisms in a category.

\begin{COR}
  Let $\CC$ be a locally small category such that all the morphisms in $\CC$ are mono.
  Assume also that $\Aut(A)$ is finite for all $A \in \Ob(\CC)$. Then $\CC$ has the Ramsey
  property for objects if and only if $t^{\mor}_\CC(A) = |\Aut(A)|$ for all $A \in \Ob(\CC)$.
\end{COR}

\noindent
This is an immediate consequence of Proposition~\ref{rdbas.prop.sml} but we believe that it is worth noting.
Since it is always the case that $t^{\mor}_\CC(A) \ge |\Aut(A)|$ (see \cite{masul-dual-kpt} for details),
it follows that the Ramsey property for objects corresponds to the situation where $t^{\mor}_\CC(A)$ attains
its minimal value for each $A \in \Ob(\CC)$.

Let $\CC$ be a locally small category.
For $A, S \in \Ob(\CC)$ let $T_\CC(A, S)$ denote the least positive integer $n$ such that
for all $k \ge 2$ we have that $S \longrightarrow (S)^{A}_{k, n}$, if such an integer exists.
Otherwise put $T_\CC(A, S) = \infty$.
Analogously, let $T^{\mor}_\CC(A, S)$ denote the least positive integer $n$ such that
for all $k \ge 2$ we have that $S \overset{\mor}\longrightarrow (S)^{A}_{k, n}$, if such an integer exists.
Otherwise put $T^{\mor}_\CC(A, S) = \infty$.

In full analogy to Proposition~\ref{rdbas.prop.sml} we now have (see \cite{Zucker-2} for the proof in case of relational
structures):

\begin{PROP}\label{rdbas.prop.big}
  Let $\CC$ be a locally small category and let $A, S \in \Ob(\CC)$ be chosen so that all the morphisms in
  $\hom_\CC(A, S)$ are mono. Then $T^{\mor}_\CC(A, S)$ is finite if and only if both $\Aut(A)$ and $T_\CC(A, S)$ are finite,
  and in that case
  $$
    T^{\mor}_\CC(A, S) = |\Aut(A)| \cdot T_\CC(A, S).
  $$
\end{PROP}
\begin{proof}
  Assume, first, that $|\Aut(A)| = \infty$. Let us show that $T^{\mor}_\CC(A, S) = \infty$ by showing that
  $T^{\mor}_\CC(A, S) \ge n$ for every $n \in \NN$. Fix an $n \in \NN$ and $X \subseteq \Aut(A)$ such that
  $|X| = n$.
  Let $\binom SA = \hom(A, S) / \Boxed{\sim_A}  = \{H_i : i \in I\}$ for some index set $I$.
  For each $i \in I$ choose a representative $h_i \in H_i$. Then $H_i = h_i \cdot \Aut(A)$.
  Fix an arbitrary $\xi \in X$ and define $\chi' : \hom(A, S) \to X$ as follows:
  \begin{itemize}
  \item[]
    if $g = h_i \cdot \alpha$ for some $i \in I$ where $\alpha \in X$ then $\chi'(g) = \alpha$;
  \item[]
    otherwise $\chi'(g) = \xi$.
  \end{itemize}
  Take any $w : S \to S$. Let $f \in \hom(A, S)$ be arbitrary. Then:
  $$
    |\chi'(w \cdot \hom(A, S))| \ge |\chi'(w \cdot f \cdot \Aut(A))|.
  $$
  Clearly, $w \cdot f \cdot \Aut(A) = h_i \cdot \Aut(A)$ for some $i \in I$, so
  $$
    |\chi'(w \cdot \hom(A, S))| \ge |\chi'(h_i \cdot \Aut(A))| = n.
  $$
  This completes the proof in case $\Aut(A)$ is infinite.

  Let us now move on to the case when $\Aut(A)$ is finite.

  Let $T_\CC(A, S) = n$ for some $n \in \NN$.
  Take any $k \ge 2$. Since $T_\CC(A, S) = n$ we have that
  $S \longrightarrow (S)^{A}_{2^k, n}$. Let $\chi : \hom(A, S) \to k$ be an
  arbitrary coloring. Construct $\chi' : \binom SA \to \calP(k)$ as follows:
  $$
    \chi'(f / \Boxed{\sim_A}) = \chi(f / \Boxed{\sim_A})
  $$
  (here, $\chi$ is applied to a set of morphisms to produce a set of colors, which is an element of $\calP(k)$).
  Then $S \longrightarrow (S)^{A}_{2^k, n}$ implies that there exists a $w : S \to S$ such that
  $|\chi'(w \cdot \binom SA)| \le n$. But then it is easy to see that
  $|\chi'(w \cdot \binom SA)| \le n$ implies $|\chi(w \cdot \hom(A, S))| \le n \cdot |\Aut(A)|$, proving
  thus that $T^\mor_\CC(A, S) \le n \cdot |\Aut(A)| = T_\CC(A, S) \cdot |\Aut(A)|$.

  For the other inequality note that $T_\CC(A, S) = n$ also implies that there is a $k \ge 2$ and
  a coloring $\chi : \binom SA \to k$ with the property that for every $w \in \hom(S, S)$ we have that
  $|\chi(w \cdot \binom SA)| \ge n$. Let $\ell = k \cdot |\Aut(A)|$.
  Let $\binom SA = \hom(A, S) / \Boxed{\sim_A}  = \{H_i : i \in I\}$ for some index set $I$.
  For each $i \in I$ choose a representative $h_i \in H_i$. Then $H_i = h_i \cdot \Aut(A)$.  
  Since all the morphisms in $\hom_\CC(A, S)$ are mono, for each $f \in \hom(A, S)$ there is a unique
  $i \in I$ and a unique $\alpha \in \Aut(A)$ such that $f = h_i \cdot \alpha$. Let us denote this
  $\alpha$ by $\alpha(f)$. Consider the following coloring:
  $$
    \xi : \hom(A, S) \to k \times \Aut(A) : f \mapsto (\chi(f/\Boxed{\sim_A}), \alpha(f))
  $$
  and take any $w \in \hom(S, S)$. Since
  $|\chi(w \cdot \binom SA)| \ge n$, it easily follows that
  $|\xi(w \cdot \hom(A, S))| \ge n \cdot |\Aut(A)|$
  proving that $T^\mor_\CC(A, S) \ge n \cdot |\Aut(A)| = T_\CC(A, S) \cdot |\Aut(A)|$.

  Assume, finally, that $T_\CC(A, S) = \infty$ and let us show that $T^{\mor}_\CC(A, S) = \infty$ by showing that
  $T^{\mor}_\CC(A, S) \ge n$ for every $n \in \NN$. Fix an $n \in \NN$. Since $T_\CC(A, S) = \infty$,
  there is a $k \ge 2$ and a coloring
  $\chi : \binom SA \to k$ such that for every $w : S \to S$ we have that
  $|\chi(w \cdot \binom SA)| \ge n$. Then the coloring $\chi' : \hom(A, S) \to k$ defined by
  $$
    \chi'(f) = \chi(f / \Boxed{\sim_A})
  $$
  has the property that $|\chi(w \cdot \hom(A, S))| \ge n$.

  This completes the proof.
\end{proof}

As an immediate corollary we have the following:

\begin{COR}
  Let $\CC$ be a locally small category and let $A, S \in \Ob(\CC)$ be chosen so that all the morphisms in
  $\hom_\CC(A, S)$ are mono. Then

  $(a)$ $T^{\mor}_\CC(A, S) \ge |\Aut(A)|$;
  
  $(b)$ if $T^{\mor}_\CC(A, S) \le n$ then $|\Aut(A)| \le n$;
  
  $(c)$ if $T^{\mor}_\CC(A, S) = 1$ then $A$ is rigid.
  
  $(d)$ if $\Aut(A)$ is finite then $T_\CC(A, S) = 1$ if and only if $T^\mor_\CC(A, S) = |\Aut(A)|$.
\end{COR}
\begin{proof}
  $(a)$ and $(d)$ follow immediately from Proposition~\ref{rdbas.prop.big}, while
  $(b)$ and $(c)$ are direct consequences of $(a)$.
\end{proof}

\section{Small Ramsey degrees as minima of the big ones}
\label{rdbas.sec.min}

It was shown in \cite{Zucker-2} that small Ramsey degrees are not greater than the corresponding big Ramsey degrees.
We shall prove a generalization of this result as Proposition~\ref{rdbas.prop.smaller} below. However,
by moving from classes of structures to general categories we can prove much more. We can show that
small Ramsey degrees are minima of the corresponding big ones. More precisely, in this section we prove the following:

\begin{THM}\label{rdbas.thm.minT}
  Let $\CC$ be a directed small category whose morphisms are mono and such that $\hom_\CC(A, B)$ is finite for
  all $A, B \in \Ob(\CC)$. Then for every $A \in \Ob(\CC)$,
  $$
    t^\mor_\CC(A) = \min_{\SS, \; S} \; T^\mor_\SS(A, S),
  $$
  where the minimum is taken over all locally small categories $\SS$ which contain $\CC$ as a full subcategory,
  and all $S \in \Ob(\SS)$ which are universal and locally finite for~$\CC$. Consequently, for every $A \in \Ob(\CC)$,
  $$
    t_\CC(A) = \min_{\SS, \; S} \; T_\SS(A, S),
  $$
  where the minimum is taken as above.
\end{THM}

We start by showing that small Ramsey degrees are indeed smaller.
Let $\DD$ be a full subcategory of a locally small category $\CC$.
An $S \in \Ob(\CC)$ is \emph{locally finite} for $\DD$ if
for every $A, B \in \Ob(\DD)$ and every
$e : A \to S$, $f : B \to S$ there exist $D \in \Ob(\DD)$,
$r : D \to S$, $p : A \to D$ and $q : B \to D$
such that $r \cdot p = e$ and $r \cdot q = f$:
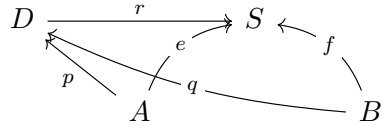
\begin{center}
  \begin{tikzcd}
    D \arrow[rr, "r"] & & S & &  \\
    & A \arrow[ul, "p"] \arrow[ur, bend left, "e" description] & & B \arrow[ulll, bend left=10, "q" description] \arrow[ul, bend right, "f" description]
  \end{tikzcd}
\end{center}
and for every $H \in \Ob(\DD)$, $r' : H \to S$,
$p' : A \to H$ and $q' : B \to H$
such that $r' \cdot p' = e$ and $r' \cdot q' = f$ there is an
$s : D \to H$ such that the diagram below commutes
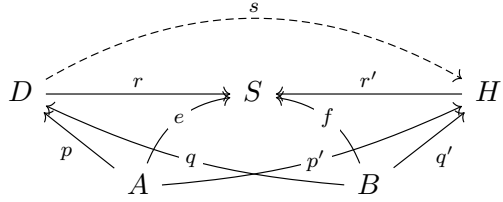
\begin{center}
  \begin{tikzcd}
    D \arrow[rr, "r"] \arrow[rrrr, bend left, dashed, "s"] & & S & & H \arrow[ll, "r'"'] \\
    & A \arrow[ul, "p"] \arrow[ur, bend left, "e" description] \arrow[urrr, bend right=10, "p'" description] & & B \arrow[ulll, bend left=10, "q" description] \arrow[ul, bend right, "f" description] \arrow[ur, "q'"']
  \end{tikzcd}
\end{center}
The motivation for this notion comes from model theory where we say that a first-order structure
$A$ is locally finite if every substructure generated by a finite set has to be finite.
The substructure generated by a subset of $A$ is the smallest
substructure of $A$ that contains the set. Now, think of $\DD$ as a category of objects of $\CC$
that we think of as ``finite''. Then $S$ is locally finite for $\DD$ if every pair of ``finite'' subobjects of $S$ is
contained in a ``finite'' subobject of $S$, and there is a smallest one with this property.

\begin{LEM}[\cite{masul-dual-kpt}]\label{akpt.lem.ramseyF}
  Let $\DD$ be a full subcategory of a locally small category $\CC$ such that $\hom(A, B)$ is finite for all $A, B \in \Ob(\DD)$,
  and let $S \in \Ob(\CC)$ be a universal and locally finite object for $\DD$. Let $k \ge 2$ and $t \ge 1$ be integers and
  $A, B \in \Ob(\DD)$ such that $A \toCC{\DD} B$.
  There is a $C \in \Ob(\DD)$ such that $C \overset\mor\longrightarrow (B)^{A}_{k, t}$
  if and only if $S \overset\mor\longrightarrow (B)^{A}_{k, t}$.
\end{LEM}

\begin{PROP}\label{rdbas.prop.smaller}
  Let $\DD$ be a full subcategory of a locally small category $\CC$ such that $\hom(A, B)$ is finite for all $A, B \in \Ob(\DD)$
  and let $S$ be a universal and locally finite object for $\DD$. Then for every $A \in \Ob(\DD)$,
  $$
    t^{\mor}_\DD(A) \le T^{\mor}_\CC(A, S),
  $$
  and consequently,
  $$
    t_\DD(A) \le T_\CC(A, S).
  $$
\end{PROP}
\begin{proof}
  Let $T^{\mor}_\CC(A, S) = n \in \NN$. To show that $t^{\mor}_\DD(A) \le n$ take any $B \in \Ob(\DD)$ and any $k \ge 2$.
  Since $S \overset\mor\longrightarrow (S)^A_{k, n}$ i $B \toCC{\CC} S$
  (because $S$ is universal for $\DD$) it easily follows that
  $S \overset\mor\longrightarrow (B)^A_{k, n}$, and by Lemma~\ref{akpt.lem.ramseyF}
  there is a $C \in \Ob(\DD)$ such that $C \overset\mor\longrightarrow (B)^{A}_{k, n}$.

  The second statement is a consequence of Propositions~\ref{rdbas.prop.sml} and~\ref{rdbas.prop.big}
  and the fact that $\Aut_\DD(A) = \Aut_\CC(A)$ because $\DD$ is a full subcategory of~$\CC$.
  (Recall that the definition of the object universal for a subcategory implies that
  all the morphisms in $\DD$ are mono, and that all the morphisms in $\hom_\CC(A, S)$ are mono,
  so the two propositions apply.)
\end{proof}

Let us now present a construction that we refer to as the \emph{power construction} for reasons
that will become apparent immediately. For a directed small category $\CC$ whose morphisms are mono
let $\Sub(\CC)$ denote the small category whose objects are all full subcategories
of $\CC$ and whose morphisms are defined as follows. For full subcategories $\AA$ and $\BB$ of $\CC$,
a morphism from $\AA$ to $\BB$ is any family $(f_A)_{A \in \Ob(\AA)}$ of $\CC$-morphisms
indexed by objects of $\AA$ where each $f_A$ is a $\CC$-morphism from $A$ to some object in~$\BB$.
In other words, $\dom(f_A) = A$ and $\cod(f_A) \in \Ob(\BB)$.
The identity morphism is $\id_\AA = (\id_A)_{A \in \Ob(\AA)}$ and the composition is straightforward:
for $(f_A)_{A \in \Ob(\AA)} : \AA \to \BB$ and $(g_B)_{B \in \Ob(\BB)} : \BB \to \DD$ the composition
$(h_A)_{A \in \Ob(\AA)} : \AA \to \DD$ is defined by $h_A = g_{\cod(f)} \cdot f_A$.

Each $A \in \Ob(\CC)$ gives rise to a subcategory $\scat A \in \Sub(\CC)$ which is the full subcategory
of $\CC$ spanned by the single object~$A$. It is easy to see that
$$
  \hom_{\Sub(\CC)}(\scat A, \scat B) = \{ (f) : f \in \hom_\CC(A, B) \}
$$
where on the right we have a set of one-element families of morphisms.
The functor
$$
  \CC \to \Sub(\CC) : A \mapsto \scat A : f \mapsto (f)
$$
is clearly an embedding. Moreover it embeds $\CC$ into $\Sub(\CC)$ ``canonically'', so in future
we shall not distinguish between $A$ and its image $\scat A$, and between
$f$ and $(f)$. We shall simply take that $\CC$ is a full subcategory of $\Sub(\CC)$
via the canonical embedding.

Note that $\CC$, being a full subcategory of itself, is also an object of $\Sub(\CC)$. Moreover,
$\CC$ as an object of $\Sub(\CC)$ is universal for $\CC$ as a full subcategory
of $\Sub(\CC)$ because all the hom-sets $\hom_{\Sub(\CC)}(A, \CC)$ are nonempty and
each morphism in $\hom_{\Sub(\CC)}(A, \CC)$ is mono in $\Sub(\CC)$, which is
easy to check.

Let us now show that both big and small Ramsey degrees in $\CC$ can be represented by
big Ramsey degrees in $\Sub(\CC)$ as follows.

\begin{LEM}
  Let $\CC$ be a small category such that all the morphisms in $\CC$ are mono, and let $A, S \in \Ob(\CC)$. Then
  $$
    T^{\mor}_\CC(A, S) = T^{\mor}_{\Sub(\CC)}(A, S).
  $$
  Consequently, if $\Aut(A)$ is finite then
  $$
    T_\CC(A, S) = T_{\Sub(\CC)}(A, S).
  $$
\end{LEM}
\begin{proof}
  The first statement is an immediate consequence of the fact that $\hom_{\Sub(\CC)}(A, S) = \hom_\CC(A, S)$
  and $\hom_{\Sub(\CC)}(S, S) = \hom_\CC(S, S)$.
  The second statement is a consequence of Proposition~\ref{rdbas.prop.big} and
  the fact that $\Aut_{\Sub(\CC)}(A) = \Aut_\CC(A)$.
\end{proof}

\begin{PROP}\label{rdbas.prop.sml-big}
  Let $\CC$ be a directed small category whose morphisms are mono and let $A \in \Ob(\CC)$. Then
  $$
    t^{\mor}_\CC(A) = T^{\mor}_{\Sub(\CC)}(A, \CC).
  $$
  Consequently, if $\Aut(A)$ is finite,
  $$
    t_\CC(A) = T_{\Sub(\CC)}(A, \CC).
  $$
\end{PROP}
\begin{proof}
  The second part of the statement is an immediate consequence of the first part of the statement and
  Propositions~\ref{rdbas.prop.sml} and~\ref{rdbas.prop.big}.
  Let us show that $t^{\mor}_\CC(A) = T^{\mor}_{\Sub(\CC)}(A, \CC)$ by showing that
  $t^{\mor}_\CC(A) \le n$ if and only if $T^{\mor}_{\Sub(\CC)}(A, \CC) \le n$, for all $n \in \NN$.
  
  $(\Rightarrow)$
  Assume that $t^{\mor}_\CC(A) \le n$ and let us show that $T^{\mor}_{\Sub(\CC)}(A, \CC) \le n$.
  Take any $k \ge 2$ and any coloring $\chi : \hom_{\Sub(\CC)}(A, \CC) \to k$.
  Then
  $$
    \chi : \bigcup_{C \in \Ob(\CC)} \hom_{\CC}(A, C) \to k,
  $$
  so for each $C \in \Ob(\CC)$ let
  $$
    \chi_C = \restr{\chi}{\hom_{\CC}(A, C)} : \hom_{\CC}(A, C) \to k.
  $$
  For $\0 \ne J \subseteq k$ let $\CC_J$ be the full subcategory of $\CC$ spanned by all $B \in \Ob(\CC)$ satisfying
  the following:
  \begin{itemize}
  \item $A \toCC{\CC} B$, and
  \item there exists a $C \in \Ob(\CC)$ and an $f \in \hom_\CC(B, C)$ such that $f \cdot \hom_\CC(A, B) \subseteq \chi^{-1}(J)$.
  \end{itemize}

  \medskip

  Claim 1: Every $B \in \Ob(\CC)$ such that $A \toCC{\CC} B$ belongs to $\Ob(\CC_J)$ for some~$J$ satisfying $|J| \le n$.
  
  Take any $B \in \Ob(\CC)$ such that $A \toCC{\CC} B$. Since $t_\CC(A) \le n$ there exists a $C \in \Ob(\CC)$ such that
  $C \overset{\mor}\longrightarrow (B)^A_{k, n}$, so there is a $w \in \hom_\CC(B, C)$ such that
  $|\chi_C(w \cdot \hom_\CC(A, B))| \le n$. Hence, $B \in \Ob(\CC_J)$ for $J = \chi_C(w \cdot \hom_\CC(A, B))$.
  This completes the proof of Claim~1.

  \medskip
  
  Claim 2: There is a $J \subseteq k$ such that $|J| \le n$ and $\CC_J$ is cofinal in $\CC$.
  
  Suppose this is not the case. Then for every $\0 \ne J \subseteq k$ such that $|J| \le n$ there exists an $X_J \in \Ob(\CC)$
  such that $\hom_\CC(X_J, C) = \0$ for all $C \in \Ob(\CC_J)$. Since $\CC$ is directed, there exists a $Y \in \Ob(\CC)$
  such that $A \toCC{\CC} Y$ and $X_J \toCC{\CC} Y$ for all $\0 \ne J \subseteq k$ with $|J| \le n$.
  (Note that there are finitely many such $J$'s.) According to Claim~1 there is a
  $J' \subseteq k$ such that $|J'| \le n$ and $Y \in \Ob(\CC_{J'})$. Then $X_{J'} \toCC{\CC} Y \in \Ob(\CC_{J'})$.
  Contradiction. This proves Claim~2.
  
  \medskip
  
  So, by Claim~2 there is a $J_0 \subseteq k$ such that $|J_0| \le n$ and $\CC_{J_0}$ is cofinal in $\CC$.
  Let us now construct $\hat w = (w_B)_{B \in \Ob(\CC)} \in \hom_{\Sub(\CC)}(\CC, \CC)$ as follows.
  Take a $B \in \Ob(\CC)$.
  \begin{itemize}
  \item
    If $\hom_\CC(A, B) = \0$ put $w_B = \id_B$.
  \item
    Assume, now, that $A \toCC{\CC} B$. Since $\CC_{J_0}$ is cofinal in $\CC$
    there is a $B_0 \in \Ob(\CC_{J_0})$ and an $h : B \to B_0$.
    Then by definition of $\CC_{J_0}$ there is a $C \in \Ob(\CC)$ and an $f : B_0 \to C$ such that $f \cdot \hom_\CC(A, B_0)
    \subseteq \chi^{-1}(J_0)$. Clearly, $h \cdot \hom_\CC(A, B) \subseteq \hom_\CC(A, B_0)$,
    so $f \cdot h \cdot \hom_\CC(A, B) \subseteq f \cdot \hom_\CC(A, B_0) \subseteq \chi^{-1}(J_0)$.
    Therefore, in this case we put $w_B = f \cdot h$.
  \end{itemize}
  It is now easy to see that $\chi(\hat w \cdot \hom_{\Sub(\CC)}(A, \CC)) \subseteq J_0$, whence
  $|\chi(\hat w \cdot \hom_{\Sub(\CC)}(A, \CC))| \le |J_0| \le n$.

  \medskip

  $(\Leftarrow)$
  Assume that $t^\mor_\CC(A) \ge n$. Then there exist a $k \ge 2$ and a $B \in \Ob(\CC)$ such that
  for every $C \in \Ob(\CC)$ one can find a coloring $\chi_C : \hom_\CC(A, C) \to k$ such that
  for every $w \in \hom_\CC(B, C)$ we have that
  $$
    |\chi_C(w \cdot \hom_\CC(A, B))| \ge n.
  $$
  Define $\hat\chi : \hom_{\Sub(\CC)}(A, \CC) \to k$ by
  $$
    \hat\chi(f) = \chi_{\cod(f)}(f).
  $$
  Take any $\hat w = (w_D)_{D \in \Ob(\CC)} \in \hom_{\Sub(\CC)}(\CC, \CC)$. Then
  \begin{align*}
    |\hat\chi(\hat w \cdot \hom_{\Sub(\CC)}(A, \CC))|
    & =   \Big| \hat\chi\Big(\bigcup_{D \in \Ob(\CC)} w_D \cdot \hom_\CC(A, D) \Big) \Big|\\
    & =   \Big| \bigcup_{D \in \Ob(\CC)} \chi_{\cod(w_D)}(w_D \cdot \hom_\CC(A, D)) \Big|\\
    & \ge |\chi_{C}(w_B \cdot \hom_\CC(A, B))| \ge n,
  \end{align*}
  where $C = \cod(w_B)$. This completes the proof that $T^{\mor}_{\Sub(\CC)}(A, \CC) \ge n$.
\end{proof}

\bigskip

\begin{proof}[Proof of Theorem~\ref{rdbas.thm.minT}]
  The second part of the statement is an immediate consequence of the first part of the statement and
  Propositions~\ref{rdbas.prop.sml} and~\ref{rdbas.prop.big}.
  
  Let us prove the first part of the statement.
  
  If $t^\mor_\CC(A) = \infty$ for some $A \in \Ob(\CC)$
  then Proposition~\ref{rdbas.prop.smaller} implies that
  $T^\mor_\SS(A, S) = \infty$ for all $\SS \ge \CC$ and all $S \in \Ob(\SS)$ which are universal for~$\CC$.

  Assume, therefore, that $t^\mor_\SS(A)$ is an integer. We already know from Proposition~\ref{rdbas.prop.smaller}
  that
  $$
    t^\mor_\CC(A) \le \min_{\SS, \; S} \; T^\mor_\SS(A, S),
  $$
  while from Proposition~\ref{rdbas.prop.sml-big} we know that the minimum is attained
  for $\SS = \Sub(\CC)$ and $S = \CC$.
\end{proof}

It is important to stress that the proof of Theorem~\ref{rdbas.thm.minT} relies on
a ``synthetic example'' to show that the minimum is attained. However, in case of chains ($=$~linearly ordered sets)
we don't need a synthetic example. From the finite and the infinite version of Ramsey's theorem we have that
$t_{\Ch_\fin}(n) = 1$ and $T_\Ch(n, \omega) = 1$ for every finite chain $n$, where $\Ch_\fin$ is the category of
finite chains together with embeddings, and $\Ch$ is the category of at most countably infinite chains
together with embeddings.
It would be of interest to identify examples of this phenomenon in categories of other types of first-order structures. For example,
is there a countable graph $U$ such that $t_{\Gra_\fin}(G) = T_\Gra(G, U)$ for every finite graph $G$, where
$\Gra_\fin$ is the category of finite graphs together with embeddings, and
$\Gra$ is the category at most countably infinite graphs together with embeddings?

\section{Monotonicity of Ramsey degrees}
\label{rdbas.sec.mon}

In this section we are going to review a few facts about the monotonicity of Ramsey degrees which have been considered
in \cite{masul-bigrd,Zucker-1,Zucker-2} but follow easily from the above considerations. We are going to show that in some cases the big Ramsey degrees are
monotonous in the first argument. This immediately implies the monotonicity of the small Ramsey degrees via
Theorem~\ref{rdbas.thm.minT}. Finally, we present a sufficient condition for the big Ramsey degrees to be monotonous
in the second argument.

Let $\CC$ be a category and $A, B, S \in \Ob(\CC)$. Then
$S$ \emph{is weakly homogeneous for $(A, B)$}, if there exist
$f \in \hom(A, B)$ and $g \in \hom(S, S)$ such that $g \cdot \hom(A, S) \subseteq \hom(B, S) \cdot f$.
\begin{center}
  \begin{tikzcd}
    B \arrow[rr, bend left=20] \arrow[rr, bend right=20] \arrow[ddrr, phantom, "\supseteq" description] & & S \\
    & \\
    A \arrow[uu, "f"] \arrow[rr, bend left=20] \arrow[rr, bend right=20] & & S \arrow[uu, "g"']
  \end{tikzcd}
\end{center}

Note that this is a weak form of weak homogeneity. An object $S \in \Ob(\CC)$ is \emph{weakly homogeneous}
for a full subcategory $\DD$ of $\CC$ if for any $A, B \in \Ob(\DD)$, any $f \in \hom_\DD(A, B)$ and any
$g \in \hom_\CC(A, S)$ there is an $h \in \hom_\CC(B, S)$ such that
\begin{center}
  \begin{tikzcd}
    B \arrow[dr, "h"] \\
    A \arrow[u, "f"] \arrow[r, "g"'] & S
  \end{tikzcd}
\end{center}
Clearly, if $S$ is weakly homogeneous for $\DD$ then $S$ is weakly homogeneous for every pair $(A, B)$ where
$A, B \in \Ob(\DD)$ such that $A \toCC{\DD} B$ because $\id_S \cdot \hom_\CC(A, S) = \hom_\CC(B, S) \cdot f$
for any $f \in \hom_\DD(A, B)$.

\begin{THM}\label{rdbas.thm.mon-bigdeg}
  Let $\CC$ be a locally small category whose morphisms are mono and let $A, B \in \Ob(\CC)$ be such that $A \toCC{\CC} B$. Then
  $T^{\mor}_\CC(A, S) \le T^{\mor}_\CC(B, S)$ for every $S \in \Ob(\CC)$ which is weakly homogeneous for $(A, B)$.
\end{THM}
\begin{proof}
  Take any $S$ which is weakly homogeneous for $(A, B)$. Then there exist
  $f \in \hom(A, B)$ and $g \in \hom(S, S)$ such that $g \cdot \hom(A, S) \subseteq \hom(B, S) \cdot f$.
  Let $T^\mor_\CC(B, S) = n \in \NN$.
  
  Take any $k \ge 2$ and let $\chi : \hom(A, S) \to k$ be a coloring.
  Define $\chi' : \hom(B, S) \to k$ by $\chi'(h) = \chi(h \cdot f)$.
  Then there is a $w \in \hom(S, S)$ such that
  $|\chi'(w \cdot \hom(B, S))| \le n$. The definition of $\chi'$ then yields
  $|\chi(w \cdot \hom(B, S) \cdot f)| \le n$. Therefore,
  $|\chi(w \cdot g \cdot \hom(A, S))| \le n$ because $g \cdot \hom(A, S) \subseteq \hom(B, S) \cdot f$.
\end{proof}

\begin{LEM}\label{rdbas.lem.slide}
  Let $\CC$ be a small category with amalgamation and $A, B \in \Ob(\CC)$. If $A \toCC{\CC} B$
  then $\CC$ (as an object of $\Sub(\CC)$) is weakly homogeneous for $(A, B)$ in $\Sub(\CC)$.
\end{LEM}
\begin{proof}
  Fix arbitrary $f \in \hom_\CC(A, B)$. Then $f \in \hom_{\Sub(\CC)}(A, B)$.
  Taking $\CC$ for $S$ in the definition of being weakly homogeneous for a pair,
  we shall now construct a morphism $(g_C)_{C \in \Ob(\CC)} : \CC \to \CC$ by amalgamation.
  Take any $C \in \Ob(\CC)$. If $\hom_\CC(A, C) = \0$ put $g_C = \id_C$. Otherwise,
  take any $h \in \hom_\CC(A, C)$. Then there is a $C' \in \Ob(\CC)$ and
  morphisms $h' \in \hom_\CC(B, C')$ and $f' \in \hom_\CC(C, C')$ such that
  \begin{center}
    \begin{tikzcd}
      B \arrow[r, "h'"] & C' \\
      A \arrow[u, "f"]  \arrow[r, "h"'] & C \arrow[u, "f'"']
    \end{tikzcd}
  \end{center}
  Put $g_C = f'$. Now it is easy to see that
  $$
    (g_C)_{C \in \Ob(\CC)} \cdot \hom_{\Sub(\CC)}(A, \CC) \subseteq \hom_{\Sub(\CC)}(B, \CC) \cdot f
  $$
  having in mind that
  $$
    \hom_{\Sub(\CC)}(A, \CC) = \bigcup_{C \in \Ob(\CC)} \hom_\CC(A, C),
  $$
  and the same for $\hom_{\Sub(\CC)}(B, \CC)$.
\end{proof}

\begin{THM}\label{rdbas.thm.mon-smldeg}
  Let $\CC$ be a directed small category with amalgamation whose morphisms are mono.
  If $A \toCC{\CC} B$ then $t^{\mor}_\CC(A) \le t^{\mor}_\CC(B)$,
  for all $A, B \in \Ob(\CC)$. 
\end{THM}
\begin{proof}
  (cf.\ \cite{Zucker-1})
  By Proposition~\ref{rdbas.prop.sml-big} it suffices to show that
  $$
    T^{\mor}_{\Sub(\CC)}(A, \CC) \le T^{\mor}_{\Sub(\CC)}(B, \CC).
  $$
  From Lemma~\ref{rdbas.lem.slide} we know that
  $\CC$ is weakly homogeneous for $(A, B)$ in $\Sub(\CC)$.
  The claim now follows from Theorem~\ref{rdbas.thm.mon-bigdeg}.
\end{proof}

Therefore, small Ramsey degrees are monotonous:
$A \toCC{\CC} B$ implies $t^\mor_\CC(A) \le t^\mor_\CC(B)$.
We have also seen (Theorem~\ref{rdbas.thm.mon-bigdeg}) that under some reasonable assumptions
big Ramsey degrees are monotonous in the first argument:
if $A \toCC{\CC} B$ and $S$ is weakly homogeneous for $(A, B)$ then $T^\mor_\CC(A, S) \le T^\mor_\CC(B, S)$. As the
following example shows the big Ramsey degrees are not necessarily monotonous in the second argument.

\begin{EX}
  Recall that a chain is a structure $(A, \Boxed<)$ where $<$ is a linear order on~$A$.
  For the sake of this example let $n$ denote the finite chain $0 < 1 < \ldots < n-1$, let $\QQ$ be the
  chain of the rationals with respect to the usual ordering, and let $\omega$ be the first infinite ordinal.
  The infinite version of Ramsey's theorem actually claims that $T(n, \omega) = 1$ for all $n \ge 1$.
  In an attempt to generalize Ramsey's theorem to other chains Galvin observed in
  \cite{galvin1,galvin2} that $T(2, \QQ) = 2$.
  This observation was later generalized by Devlin in \cite{devlin} who showed that $T(n, \QQ) < \infty$
  for all $n \ge 2$, and was actually able to compute the exact values of $T(n, \QQ)$.

  In \cite{masul-sobot} the authors made another step towards computing the big Ramsey degrees in various ordinals.
  For example, they were able to show that $T(n, \omega \cdot m) = m^n$, while
  $T(n, \omega^\omega) = \infty$ for all $n \ge 2$ (where $\omega^\omega$ in this context denotes the
  ordinal exponentiation; hence $\omega^\omega$ is a countable chain).
  
  Fix an $n \in \NN$ and take $m \in \NN$ so that $m^n > T(n, \QQ)$. Then $\omega \cdot m$ embeds into $\QQ$ but
  $T(n, \omega \cdot m) > T(n, \QQ)$. Moreover, for any $n \ge 2$ we have that $\omega^\omega$ embeds into $\QQ$ but
  $T(n, \omega^\omega) = \infty > T(n, \QQ)$.
\end{EX}

Nevertheless, under certain assumptions the big Ramsey degrees are monotonous in the second argument as well.
One such situation was identified in~\cite{masul-bigrd} as follows and we shall get back to it in
Section~\ref{rdbas.sec.appls}.

Consider an acyclic, bipartite, not necessarily finite digraph where all the arrows go from one class of vertices into the other
and the out-degree of all the vertices in the first class is~2:
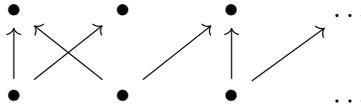
\begin{center}
  \begin{tikzcd}
    \bullet & \bullet & \bullet & \dots \\
    \bullet \arrow[u] \arrow[ur] & \bullet \arrow[ur] \arrow[ul] & \bullet \arrow[u] \arrow[ur] & \dots 
  \end{tikzcd}
\end{center}
\noindent
Such a digraph will be referred to as a \emph{binary digraph}.
Let $\CC$ be a category. For $A, B \in \Ob(\CC)$, an \emph{$(A, B)$-diagram} in a category $\CC$ is a functor
$F : \Delta \to \CC$ where $\Delta$ is a binary digraph,
$F$ takes the bottom row of $\Delta$ onto $A$, and takes the top row of $\Delta$ onto $B$, Fig.~\ref{nrt.fig.2}.

\begin{figure}
  \centering
  \begin{tikzcd}
    \bullet & \bullet & \bullet
    & & B & B & B
  \\
        \bullet \arrow[u] \arrow[ur, bend right]                & \bullet \arrow[ur, bend right] \arrow[ul, bend left]                        & \bullet \arrow[ul, bend left] \arrow[u]
    & & A \arrow[u, "f_1"] \ar[ur, "f_2", bend right, near end] & A \ar[ur, "f_4", bend right, near end] \ar[ul, "f_3"', bend left, near end] & A \ar[ul, "f_5"', bend left, near end] \ar[u, "f_6"']
  \\
    & \Delta \arrow[rrrr, "F"]  & & & & \CC  
  \end{tikzcd}
  \caption{An $(A, B)$-diagram in $\CC$}
  \label{nrt.fig.2}
\end{figure}

\begin{figure}
  \centering
  \begin{tikzcd}[execute at end picture={
            \draw (-0.5,-1.5) rectangle (5.75,1);
        }]
    & & & & & C & & \CC
  \\
    & & & & &  &
  \\
    \bullet & \bullet & \bullet
    & & B \arrow[uur] \arrow[rrr, dotted, bend left=14] & B \arrow[uu] \arrow[rr, dotted, bend left=10] & B \arrow[uul] \arrow[r, dotted] & D
  \\
    \bullet \arrow[u] \arrow[ur] & \bullet \arrow[ur] \arrow[ul] & 
    & & A \arrow[u] \arrow[ur] & A \arrow[ur] \arrow[ul] & & \BB
  \\
    & \Delta \arrow[rrrr, "F"]  & & & & \BB
  \end{tikzcd}
  \caption{The setup of Theorem~\ref{bigrd.thm.1}}
  \label{bigrd.fig.subcat}
\end{figure}

\begin{THM}\label{bigrd.thm.1}\cite{masul-bigrd}
  Let $\CC$ be a locally small category whose morphisms are mono and let $\BB$ be a (not necessarily full) subcategory of $\CC$.
  Let $B \in \Ob(\BB)$ be universal for $\BB$ and let $C \in \Ob(\CC)$ be universal for $\CC$.
  Take any $A \in \BB$ and assume that for every $(A, B)$-diagram $F : \Delta \to \BB$ the following holds:
  if $F$ (which is an $(A, B)$-diagram in $\CC$ as well)
  has a commuting cocone in $\CC$ whose tip is $C$, then $F$ has a commuting cocone
  in~$\BB$, Fig.~\ref{bigrd.fig.subcat}. Then $T_\BB(A, B) \le T_\CC(A, C)$.
\end{THM}

\section{An additive property of big Ramsey degrees}
\label{rdbas.sec.additivity-big}

In this section we refine a result from~\cite{dasilvabarbosa} about the additivity of big Ramsey degrees.
We prove that big Ramsey degrees for morphisms as well as big Ramsey degrees for objects posses an additive property.
Moreover, the requirement that the expansion be reasonable may be omitted.
This will have significant consequences in Section~\ref{rdbas.sec.appls}.

\begin{THM}\label{sbrd.thm.big1}
  Let $\CC$ and $\CC^*$ be locally small categories.
  Let $U : \CC^* \to \CC$ be an expansion with restrictions and assume that all the morphisms in $\CC$
  are mono. Let $S^* \in \Ob(\CC^*)$ be universal for $\CC^*$ and let $S = U(S^*)$.
  Then $S$ is (clearly) universal for $\CC$ and
  $$
    T^{\mor}_{\CC}(A, S) \le \sum_{A^* \in U^{-1}(A)} T^{\mor}_{\CC^*}(A^*, S^*).
  $$
  Consequently, if $U^{-1}(A)$ is finite and
  $T^{\mor}_{\CC^*}(A^*, S^*) < \infty$ for all $A^* \in U^{-1}(A)$ then $T^{\mor}_\CC(A, S) < \infty$.
\end{THM}
\begin{proof}
  If there is an $A^* \in U^{-1}(A)$ with $T^{\mor}_{\CC^*}(A^*, S^*) = \infty$ then the inequality is trivially satisfied. The same
  holds if $U^{-1}(A)$ is infinite. Assume, therefore, that $U^{-1}(A) = \{A^*_1, A^*_2, \ldots, A^*_n\}$ and let
  $T^{\mor}_{\CC^*}(A^*_i, S^*) = T_i \in \NN$ for each~$i$.
  
  For an arbitrary $k \ge 2$ let us show that
  $$
    S \overset{\mor}\longrightarrow (S)^A_{k, T_1 + \ldots + T_n}.
  $$
  Take any $\chi : \hom_\CC(A, S) \to k$. By Lemma~\ref{sbrd.lem.disj-union} we know that
  $$
    \hom_\CC(A, S) = \bigcup_{i=1}^n \hom_{\CC^*}(A^*_i, S^*),
  $$
  so we can restrict $\chi$ to each $\hom_{\CC^*}(A^*_i, S^*)$ to get $n$ colorings
  $$
    \chi_i : \hom_{\CC^*}(A^*_i, S^*) \to k, \quad \chi_i(f) = \chi(f), \quad i \in \{1, \ldots, n\}.
  $$
  Let us construct
  $$
    \chi'_i : \hom_{\CC^*}(A^*_i, S^*) \to k \quad \text{and} \quad w_i : S^* \to S^*, \quad
    i \in \{1, \ldots, n\}
  $$
  inductively as follows. First, put $\chi'_n = \chi_n$. Given $\chi'_i : \hom_{\CC^*}(A^*_i, S^*) \to k$, construct $w_i$
  by the Ramsey property: since $S^* \overset{\mor}\longrightarrow (S^*)^{A^*_i}_{k, T_i}$, there is a $w_i : S^* \to S^*$
  such that
  $$
    |\chi'_i(w_i \cdot \hom_{\CC^*}(A^*_i, S^*))| \le T_i.
  $$
  Finally, given $w_i : S^* \to S^*$ define $\chi'_{i-1} : \hom_{\CC^*}(A^*_{i-1}, S^*) \to k$ by
  $$
    \chi'_{i-1}(f) = \chi_{i-1}(w_n \cdot \ldots \cdot w_i \cdot f).
  $$
  Let us show that
  $$
    |\chi(w_n \cdot \ldots \cdot w_1 \cdot \hom_\CC(A, S))| \le T_1 + \ldots + T_n.
  $$
  By Lemma~\ref{sbrd.lem.disj-union} we know that $\hom_\CC(A, S) = \bigcup_{i=1}^n \hom_{\CC^*}(A^*_i, S^*)$, so
  \begin{align*}
    |\chi(w_n \cdot \ldots \cdot w_1 \cdot \hom_\CC(A, S))|
    & = |\chi(w_n \cdot \ldots \cdot w_1 \cdot \bigcup_{i=1}^n \hom_{\CC^*}(A^*_i, S^*))| \\
    & = |\chi(\bigcup_{i=1}^n w_n \cdot \ldots \cdot w_1 \cdot \hom_{\CC^*}(A^*_i, S^*))| \\
    & = |\bigcup_{i=1}^n \chi(w_n \cdot \ldots \cdot w_1 \cdot \hom_{\CC^*}(A^*_i, S^*))| \\
    & \le \sum_{i=1}^n |\chi(w_n \cdot \ldots \cdot w_1 \cdot \hom_{\CC^*}(A^*_i, S^*))|.
  \end{align*}
  Clearly,
  $
    w_n \cdot \ldots \cdot w_1 \cdot \hom_{\CC^*}(A^*_i, S^*) \subseteq \hom_{\CC^*}(A^*_i, S^*)
  $
  so,
  \begin{align*}
    |\chi(w_n \cdot \ldots \cdot w_1 \cdot \hom_{\CC^*}(A^*_i, S^*))|
    & = |\chi_i(w_n \cdot \ldots \cdot w_1 \cdot \hom_{\CC^*}(A^*_i, S^*))| \\
    & = |\chi'_i(w_i \cdot \ldots \cdot w_1 \cdot \hom_{\CC^*}(A^*_i, S^*))| \\
    & \le |\chi'_i(w_i \cdot \hom_{\CC^*}(A^*_i, S^*))| \le T_i.
  \end{align*}
  This completes the proof.
\end{proof}

Let $U : \CC^* \to \CC$ be an expansion with unique restrictions.
We say that $S^* \in \Ob(\CC^*)$ is \emph{self-similar} if the following holds:
for every $w \in \hom_\CC(S, S)$ there is a morphism $v \in \hom_{\CC^*}(S^*, \restr{S^*}{w})$,
where $S = U(S^*)$.
\begin{center}
  \begin{tikzcd}
    S^* \arrow[r, "\exists v"] & \restr{S^*}{w} \arrow[r, "w"] \arrow[d, mapsto, dashed, "U"'] & S^* \arrow[d, mapsto, dashed, "U"]\\
       & S \arrow[r, "\forall w"] & S
  \end{tikzcd}
\end{center}

\begin{EX}
  Let $\calR = (V, E)$ be the random graph (the unique, up to isomorphism, undirected
  countable ultrahomogeneous graph which is universal for the class of all the finite and countably infinite undirected graphs),
  and let $<_\NN$ and $<_\QQ$ be two linear orders of $V$ such that $(V, \Boxed{<_\NN})$ is isomorphic to~$\NN$ as a chain,
  while $(V, \Boxed{<_\QQ})$ is isomorphic to~$\QQ$ as a chain.
  Let $\calR^* = (V, E, \Boxed{<_\NN})$ and $\calR^{**} = (V, E, \Boxed{<_\QQ})$. It is easy to see that
  $\calR^*$ is self-similar, because for every embedding $w : \calR \hookrightarrow \calR$ the induced substructure
  $\restr{\calR^*}w$ contains a copy of~$\calR^*$. On the other hand, $\calR^{**}$ is not self-similar.
  To see why, note, first, that it is easy to find an embedding $w : \calR \hookrightarrow \calR$ such that
  the induced substructure $\restr{\calR^{**}}w$ is isomorphic to $\calR^*$. Therefore,
  $\restr{\calR^{**}}w$ cannot contain a copy of $\calR^{**}$ because $\QQ$ does not embed into~$\NN$.
\end{EX}

\begin{LEM}\label{sbrd.lem.big-xp-main}
  Let $\CC$ and $\CC^*$ be locally small categories.
  Let $U : \CC^* \to \CC$ be an expansion with unique restrictions and assume that all the morphisms in $\CC$
  are mono. Let $S^* \in \Ob(\CC^*)$ be universal for $\CC^*$ and self-similar, and let $S = U(S^*)$.
  (Then $S$ is (clearly) universal for $\CC$.) Let $A \in \Ob(\CC)$ be arbitrary, let
  $A^*_1, \ldots, A^*_n \in U^{-1}(A)$ be distinct and assume that
  $T_i = T^{\mor}_{\CC^*}(A^*_i, S^*) \in \NN$, $i \in \{1, \ldots, n\}$. Then
  $T^{\mor}_\CC(A, S) \ge \sum_{i=1}^n T_i$.
\end{LEM}
\begin{proof}
  Since $T^{\mor}_{\CC^*}(A^*_i, S^*) = T_i$, $i \in \{1, \ldots, n\}$, for every $i \in \{1, \ldots, n\}$
  there exists a $k_i \ge 2$ and a coloring $\chi_i : \hom_{\CC^*}(A^*_i, S^*) \to k_i$ such that for every
  $u \in \hom_{\CC^*}(S^*, S^*)$ we have that $|\chi_i(u \cdot \hom_{\CC^*}(A^*_i, S^*))| \ge T_i$.
  
  Put $k = k_1 + \ldots + k_n$ and construct $\chi : \hom_\CC(A, S) \to k = k_1 + \ldots + k_n$ as follows. Having in mind
  Lemma~\ref{sbrd.lem.disj-union},
  \begin{itemize}
  \item[] for $f \in \hom_{\CC^*}(A^*_1, S^*)$ put $\chi(f) = \chi_1(f)$;
  \item[] for $f \in \hom_{\CC^*}(A^*_2, S^*)$ put $\chi(f) = k_1 + \chi_2(f)$;
  \item[] \ \vdots
  \item[] for $f \in \hom_{\CC^*}(A^*_n, S^*)$ put $\chi(f) = k_1 + \ldots + k_{n-1} + \chi_n(f)$;
  \item[] for all other $f \in \hom_\CC(A, S)$ put $\chi(f) = 0$.
  \end{itemize}

  Let $w \in \hom_\CC(S, S)$ be arbitrary. Because $U : \CC^* \to \CC$ has unique restrictions and because $S^*$ is
  self-similar there is a $v \in \hom_{\CC^*}(S^*, \restr{S^*}{w})$.
  Let us show that $|\chi(w \cdot v \cdot \hom_\CC(A, S))| \ge T_1 + \ldots + T_n$. Note, first, that
  $$
    |\chi(w \cdot v \cdot \hom_\CC(A, S))| \ge |\chi(\bigcup_{i=1}^n w \cdot v \cdot \hom_{\CC^*}(A^*_i, S^*))|.
  $$
  The sets $w \cdot v \cdot \hom_{\CC^*}(A^*_i, S^*)$, $i \in \{1, \ldots, n\}$, are pairwise
  disjoint (since $w \cdot v \cdot \hom_{\CC^*}(A^*_i, S^*) \subseteq \hom_{\CC^*}(A^*_i, S^*)$)
  and, by construction, on each of these sets $\chi$ takes disjoint sets of values (since
  $\hom_{\CC^*}(A^*_1, S^*) \subseteq \{0, \ldots, k_1-1\}$, 
  $\hom_{\CC^*}(A^*_2, S^*) \subseteq \{k_1, \ldots, k_1+k_2-1\}$, and so on). Therefore,
  $$
    |\chi(\bigcup_{i=1}^n w \cdot v \cdot \hom_{\CC^*}(A^*_i, S^*))|
    = \sum_{i=1}^n |\chi(w \cdot v \cdot \hom_{\CC^*}(A^*_i, S^*))|.
  $$
  As another consequence of the construction of $\chi$ we have that
  $$
    |\chi(w \cdot v \cdot \hom_{\CC^*}(A^*_i, S^*))| =
    |\chi_i(w \cdot v \cdot \hom_{\CC^*}(A^*_i, S^*))| \ge T_i
  $$
  for all $i \in \{1, \ldots, n\}$, which concludes the proof of the lemma.
\end{proof}

\begin{THM}\label{sbrd.thm.big2}
  Let $\CC$ and $\CC^*$ be locally small categories.
  Let $U : \CC^* \to \CC$ be an expansion with unique restrictions and assume that all the morphisms in $\CC$
  are mono. Let $S^* \in \Ob(\CC^*)$ be universal for $\CC^*$ and self-similar, and let $S = U(S^*)$.
  Then $S$ is (clearly) universal for $\CC$ and for all $A \in \Ob(\CC)$ we have that
  $$
    T^{\mor}_{\CC}(A, S) = \sum_{A^* \in U^{-1}(A)} T^{\mor}_{\CC^*}(A^*, S^*)
  $$
  
  Consequently, $T^{\mor}_\CC(A, S) < \infty$ if and only if $U^{-1}(A)$ is finite and
  $T^{\mor}_{\CC^*}(A^*, S^*) < \infty$ for all $A^* \in U^{-1}(A)$.
\end{THM}
\begin{proof}
  It suffices to show the following three facts:
  \begin{itemize}
  \item[(1)]
    if $U^{-1}(A) = \{A^*_1, A^*_2, \ldots, A^*_n\}$ is finite and
    $T^{\mor}_{\CC^*}(A^*_i, S^*) < \infty$ for all $i$, then $T^{\mor}_{\CC}(A, S) = \sum_{i=1}^n T^{\mor}_{\CC^*}(A^*_i, S^*)$;
  \item[(2)]
    if $U^{-1}(A)$ is infinite and $T^{\mor}_{\CC^*}(A^*, S^*) < \infty$ for all $A^* \in U^{-1}(A)$
    $T^{\mor}_{\CC}(A, S) = \infty$; and
  \item[(3)]
    if there exists an $A^* \in U^{-1}(A)$ such that $T^{\mor}_{\CC^*}(A^*, S^*) = \infty$ then $T^{\mor}_{\CC}(A, S) = \infty$.
  \end{itemize}

  (1)
  Assume that $U^{-1}(A) = \{A^*_1, A^*_2, \ldots, A^*_n\}$ is finite and that\break
  $T^{\mor}_{\CC^*}(A^*_i, S^*) < \infty$ for all $i$.
  We have already seen (Theorem~\ref{sbrd.thm.big1}) that $T^{\mor}_{\CC}(A, S) \le \sum_{i=1}^n T^{\mor}_{\CC^*}(A^*_i, S^*)$, and
  that $T^{\mor}_{\CC}(A, S) \ge \sum_{i=1}^n T^{\mor}_{\CC^*}(A^*_i, S^*)$ (Lemma~\ref{sbrd.lem.big-xp-main}).

  \bigskip
  
  (2)
  Assume that $U^{-1}(A)$ is infinite and that $T^{\mor}_{\CC^*}(A^*, S^*) < \infty$ for all $A^* \in U^{-1}(A)$.
  Let us show that $T^{\mor}_\CC(A, S) = \infty$ by showing that $T^{\mor}_\CC(A, S) \ge n$ for every $n \in \NN$.
  Fix an $n \in \NN$ and take $n$ distinct $A^*_1, \ldots, A^*_n \in U^{-1}(A)$.
  Then, by Lemma~\ref{sbrd.lem.big-xp-main},
  $$
    T^{\mor}_{\CC}(A, S) \ge \sum_{i=1}^n T^{\mor}_{\CC^*}(A^*_i, S^*) \ge n.
  $$
  
  \bigskip
  
  (3)
  Assume that there is an $A^* \in U^{-1}(A)$ with $T^{\mor}_{\CC^*}(A^*, S^*) = \infty$.
  Let us show that $T^{\mor}_\CC(A, S) = \infty$ by showing that $T^{\mor}_\CC(A, S) \ge n$ for every $n \in \NN$. Fix an $n \in \NN$.
  The proof is a modification of the proof of Lemma~\ref{sbrd.lem.big-xp-main}.

  Since $T^{\mor}_{\CC^*}(A^*, S^*) = \infty$ there exists a $k \ge 2$ and a 
  coloring $\chi' : \hom_{\CC^*}(A^*, S^*) \to k$ such that for every
  $u \in \hom_{\CC^*}(S^*, S^*)$ we have that $|\chi'(u \cdot \hom_{\CC^*}(A^*, S^*))| \ge n$.
  Construct $\chi : \hom_\CC(A, S) \to k$ as follows:
  \begin{itemize}
  \item[] for $f \in \hom_{\CC^*}(A^*, S^*)$ put $\chi(f) = \chi'(f)$;
  \item[] for all other $f \in \hom_\CC(A, S)$ put $\chi(f) = 0$.
  \end{itemize}
  Let $w \in \hom_\CC(S, S)$ be arbitrary. Because $S^*$ is self-similar
  there is a $v : S^* \to \restr{S^*}{w}$.

  In order to show that $|\chi(w \cdot v \cdot \hom_\CC(A, S))| \ge n$ note, first, that
  $$
    |\chi(w \cdot v \cdot \hom_\CC(A, S))| \ge |\chi(w \cdot v \cdot \hom_{\CC^*}(A^*, S^*))|.
  $$
  Since $w \cdot v \cdot \hom_{\CC^*}(A^*, S^*) \subseteq \hom_{\CC^*}(A^*, S^*)$
  we have that
  $$
    \chi(w \cdot v \cdot \hom_{\CC^*}(A^*, S^*)) = \chi'(w \cdot v \cdot \hom_{\CC^*}(A^*, S^*))
  $$
  so, by the choice of $\chi'$,
  $$
    |\chi(w \cdot v \cdot \hom_\CC(A, S))| \ge |\chi'(w \cdot v \cdot \hom_{\CC^*}(A^*, S^*))| \ge n.
  $$
  This concludes the proof.
\end{proof}

\begin{COR}\label{sbrd.cor.big2-obj}
  Let $\CC$ and $\CC^*$ be locally small categories.
  Let $U : \CC^* \to \CC$ be an expansion with unique restrictions and assume that all the morphisms in $\CC$
  are mono. Let $S^* \in \Ob(\CC^*)$ be universal for $\CC^*$ and self-similar, and let $S = U(S^*)$
  (then $S$ is (clearly) universal for $\CC$).
  
  Let $A \in \Ob(\CC)$ be such that $\Aut(A)$ is finite.
  
  $(a)$
  $T_{\CC}(A, S)$ is finite if and only if $U^{-1}(A)$ is finite and $T_{\CC^*}(A^*, S^*) < \infty$ for all $A^* \in U^{-1}(A)$, and in that case
  $$
    T_{\CC}(A, S) = \sum_{A^* \in U^{-1}(A)} \frac{|\Aut(A^*)|}{|\Aut(A)|} \cdot T_{\CC^*}(A^*, S^*).
  $$
  
  $(b)$
  Assume that $U^{-1}(A)$ is finite and $T_{\CC^*}(A^*, S^*) < \infty$ for all $A^* \in U^{-1}(A)$.
  Let $A^*_1$, \ldots, $A^*_n$ be representatives of isomorphism classes of objects in $U^{-1}(A)$.
  Then
  $$
    T_{\CC}(A, S) = \sum_{i=1}^n T_{\CC^*}(A^*_i, S^*).
  $$
\end{COR}
\begin{proof}
  $(a)$
  Since $\Aut(A)$ is finite, Proposition~\ref{rdbas.prop.big} implies that
  $T_{\CC}(A, S)$ is finite if and only if $T^{\mor}_{\CC}(A, S)$ is finite.
  Moreover, $\Aut(A^*)$ is finite for all $A^* \in U^{-1}(A)$ because $\Aut(A^*) \subseteq \Aut(A)$.
  
  $(\Leftarrow)$
  Assume, first, that $T_\CC(A, S)$ is not finite. Then $T^{\mor}_{\CC}(A, S)$ is not finite, so by Theorem~\ref{sbrd.thm.big2},
  $U^{-1}(A)$ is not finite or there is an $A^* \in U^{-1}(A)$ such that $T^{\mor}_{\CC^*}(A^*, S^*)$ is not finite.
  The remark at the beginning of the proof then implies that $U^{-1}(A)$ is not finite or there is an $A^* \in U^{-1}(A)$
  such that $T_{\CC^*}(A^*, S^*)$ is not finite.

  $(\Rightarrow)$
  Assume, now, that $T_\CC(A, S)$ is finite. Then $T^{\mor}_{\CC}(A, S)$ is finite, so by Theorem~\ref{sbrd.thm.big2},
  $U^{-1}(A)$ is finite, say $U^{-1}(A) = \{A^*_1, \ldots, A^*_n\}$, and
  $$
    T^{\mor}_{\CC}(A, S) = \sum_{i=1}^n T^{\mor}_{\CC^*}(A^*_i, S^*).
  $$
  By Proposition~\ref{rdbas.prop.big} we get:
  $$
    |\Aut(A)| \cdot T_{\CC}(A, S) = \sum_{i=1}^n |\Aut(A^*_i)| \cdot  T_{\CC^*}(A^*_i, S^*),
  $$
  whence the claim of the corollary follows after dividing by $|\Aut(A)|$.
  
  $(b)$
  By the assumption, $U^{-1}(A) / \Boxed\cong = \{A^*_1 / \Boxed\cong, \ldots, A^*_n / \Boxed\cong \}$. Then
  \begin{align*}
    T_{\CC}(A, S)
    &= \sum_{A^* \in U^{-1}(A)} \frac{|\Aut(A^*)|}{|\Aut(A)|} \cdot T_{\CC^*}(A^*, S^*) && \text{by } (a)\\
    &= \sum_{i=1}^n |A^*_i / \Boxed\cong| \cdot \frac{|\Aut(A^*_i)|}{|\Aut(A)|} \cdot T_{\CC^*}(A^*_i, S^*) \\
    &= \sum_{i=1}^n T_{\CC^*}(A^*_i, S^*) && \text{by Lemma~\ref{sbrd.lem.iso-disj-union}~$(b)$}.\qedhere
  \end{align*}
\end{proof}

\section{Reducts of relational structures}
\label{rdbas.sec.appls}

In this section we apply the abstract machinery developed in the paper to show that
if a countably infinite relational structure has finite big Ramsey degrees, then so do its quantifier-free reducts.
Moreover, we prove that if an ultrahomogeneous countably infinite structure has finite big Ramsey degrees, then so
does the structure obtained from it by adding finitely many constants.
In particular, it follows that the reducts of $(\QQ, \Boxed<)$, the random graph, the random tournament, the random ordered graph,
$(\QQ, \Boxed<, 0)$ and all of their quantifier-free reducts
have finite big Ramsey degrees. The strategy we use is analogous to the one
used in \cite{dasilvabarbosa} to prove that the local orders $\mathbf{S}(n)$ have finite big Ramsey degrees.

A \emph{relational language} is a first-order language $L$ consisting of finitary relational symbols.
An \emph{$L$-structure} $\calA = (A, L^\calA)$ is a set $A$ together with a set $L^\calA$ of finitary relations
on $A$ which are the interpretations of the corresponding symbols in $L$.
An \emph{embedding} $f: \calA \hookrightarrow \calB$ between two $L$-structures
is an injective map $f: A \rightarrow B$ such that
  for every $R \in L$ we have that
  $(a_1, \dots, a_r) \in R^\calA \Leftrightarrow (f(a_1), \dots, f(a_r)) \in R^\calB$,
  where $r$ is the arity of~$R$.
We write $\calA \hookrightarrow \calB$ to denote that $\calA$ embeds into $\calB$, or
$f : \calA \hookrightarrow \calB$ to indicate that $f$ is an embedding.
In this section embeddings are the only structure maps we are interested in, so
a structure $\calU$ is \emph{universal for a class $\KK$} if $\calA \hookrightarrow \calU$ for every
$\calA \in \KK$.

A class $\KK$ of $L$-structures is \emph{hereditary} if the following holds:
if $\calA \in \KK$ and $\calB$ is an $L$-structure which embeds into $\calA$, then $\calB \in \KK$.

Let $L = \{ R_i : i \in I \}$.
An $L$-structure $\calA$ is a \emph{substructure} of an $L$-structure $\calB$
if $A \subseteq B$ and the identity map $a \mapsto a$ is an embedding of $\calA$ into $\calB$.
Let $\calA$ be a structure and $B \subseteq A$. Then $\calA[B]$ denotes the
\emph{substructure of $\calA$ induced by~$B$}: $\calA[B] = (B, \restr{R_i^\calA}{B})_{i \in I}$.
In case of $B = \{b_1, \ldots, b_n\}$ we also write $\calA[b_1, \ldots, b_n]$.

An $L$-structure $\calU$ is \emph{ultrahomogeneous} if for every finite $L$-structure $\calA$
and every pair of embeddings $f : \calA \hookrightarrow \calU$ and
$g : \calA \hookrightarrow \calU$ there is an automorphism $h \in \Aut(\calU)$
such that $f = h \circ g$.

Let $L = \{ R_i : i \in I \}$ and $M = \{ S_j : j \in J \}$ be relational languages. An $M$-structure
$\calA = (A, S_j^{\calA})_{j \in J}$ is a \emph{reduct} of an $L$-structure $\calA^*  = (A, R_i^{\calA^*})_{i \in I}$
if there exists a set $\Phi = \{ \phi_j : j \in J \}$ of $L$-formulas such that
for each $j \in J$ (where $\overline a$ denotes a tuple of elements of the appropriate length):
$$
  \calA \models S_j[\overline a] \text{ if and only if } \calA^* \models \phi_j[\overline a].
$$
We then say that $\calA$ is \emph{defined in $\calA^*$ by $\Phi$}.

A countably infinite relational structure may well have uncountably many distinct reducts.
However, many of those turn out to be one and the same structure presented in different languages.
Reducts $\calA_1 = (A, L_1^{\calA_1})$ and $\calA_2 =(A , L_2^{\calA_2})$ of a relational structure $\calA = (A, L^{\calA})$
are \emph{equivalent}, in symbols $\calA_1 \sim \calA_2$, if $\Aut(\calA_1) = \Aut(\calA_2)$. (The motivation comes from the
fact that if $\calA_1$ and $\calA_2$ are $\omega$-categorical structures with the same automorphism
group then each can be defined in the other by a set of first-order formulas.) We are interested in classifying
reducts of a countably infinite structure up to equivalence. Hence, representatives of equivalence classes of reducts of $\calA$
under~$\sim$ will be referred to as the \emph{essential reducts}.

Let $\KK^*$ be a class of $L$-structures and $\KK$ a class of $M$-structures. We say that
$\calA \in \KK$ is \emph{definable by $\Phi$ in $\KK^*$} if there is an $\calA^* \in \KK^*$ such that
$\calA$ is defined by $\Phi$ in $\calA^*$.

\begin{THM}\label{rdbas.thm.reducts}
  Let $L = \{ R_1, \ldots, R_n \}$ be a finite relational language, let
  $M = \{ S_j : j \in J \}$ be a relational language and let
  $\Phi = \{ \phi_j : j \in J \}$ be a set of quantifier-free $L$-formulas.
  Let $\KK^*$ be a hereditary class of at most countably infinite $L$-structures and let
  $\KK$ be the class of all the $M$-structures which are definable by $\Phi$ in~$\KK^*$.
  Moreover, let $\calS^* \in \KK^*$ be universal for $\KK^*$ and let $\calS \in \KK$ be the $M$-structure
  defined in $\calS^*$ by $\Phi$. Then
  \begin{itemize}
  \item
    $\calS$ is universal for $\KK$, and
  \item
    if $\calS^*$ has finite big Ramsey degrees, then so does $\calS$.
  \end{itemize}
\end{THM}
\begin{proof}
  We shall start by a simple but important observation. Let $\calA^*$ and $\calB^*$ be
  $L$-structures, let $\calA$ be an $M$-structure defined in $\calA^*$ by $\Phi$ and
  let $\calB$ be an $M$-structure defined in $\calB^*$ by $\Phi$. If $f$ is an embedding
  $\calA^* \hookrightarrow \calB^*$ then $f$ is also an embedding $\calA \hookrightarrow \calB$.
  This follows by a straightforward induction on the complexity the formula in question.
  Consequently, $\calS$ is universal for $\KK$ because $\calS^*$ is universal for $\KK^*$.
  
  To show that $\calS$ has finite big Ramsey degrees
  let us first note that we can understand
  $\KK$ and $\KK^*$ as categories of structures by taking embeddings as morphisms.
  Define $U : \KK^* \to \KK$ on objects by $U(\calA^*) = \mathstrut$ the $M$-structure defined in $\calA^*$ by $\Phi$,
  and on morphisms by $U(f) = f$. This is clearly an expansion. Let us show that
  $U$ has restrictions.
  
  Put $I = \{1, \ldots, n \}$.
  Let $\calA^* = (A, R^{\calA^*}_i)_{i \in I} \in \KK^*$ be arbitrary, let
  $U(\calA^*) = \calA = (A, S^{\calA}_j)_{j \in J}$, and let $f : \calB \hookrightarrow \calA$ be an
  embedding in $\KK$ where $\calB = (B, S^{\calB}_j)_{j \in J}$. By the definition of $\KK$ there is
  a $\calB^* = (B, R^{\calB^*}_i)_{i \in I} \in \KK^*$ such that $U(\calB^*) = \calB$.
  \begin{center}
    \begin{tikzcd}
      \calB_1^* = (B, R^{\calB_1^*}_i)_{i \in I} \arrow[d, mapsto, "U"] \arrow[rr, hook, bend left=10, "f"]
      & \calB^* = (B, R^{\calB^*}_i)_{i \in I} \arrow[d, mapsto, "U"]
      & \calA^* = (A, R^{\calA^*}_i)_{i \in I} \arrow[d, mapsto, "U"]
      \\
      \calB_1 = (B, S^{\calB_1}_j)_{j \in J} \arrow[r, equal, "?"] \arrow[rr, hook, bend right=10, "f"']
      & \calB = (B, S^{\calB}_j)_{j \in J} \arrow[r, hook, "f"]
      & \calA = (A, S^{\calA}_j)_{j \in J}
    \end{tikzcd}
  \end{center}
  Define $\calB_1^* = (B, R^{\calB_1^*}_i)_{i \in I}$ as follows:
  $\overline b \in R^{\calB_1^*}_i \text{ iff } f(\overline b) \in R^{\calA^*}_i$, $i \in I$.
  Then, clearly, $f : \calB_1^* \hookrightarrow \calA^*$ so $\calB_1^* \in \KK^*$ because
  $\KK^*$ is hereditary. Let $\calB_1 = (B, S^{\calB_1}_j)_{j \in J} = U(\calB_1^*)$.
  In order to complete the proof it suffices to show that $\calB_1 = \calB$. But this is immediate:
  $f$ is an embedding $\calB_1 \hookrightarrow \calA$ by the remark we made at the beginning
  of the proof; therefore, $f : \calB \hookrightarrow \calA$ and $f : \calB_1 \hookrightarrow \calA$
  whence $\calB = \calB_1$.

  For any finite $\calA^* \in \KK^*$ we know that $T_{\KK^*}(\calA^*, \calS^*) < \infty$ (by assumption),
  whence $T^\mor_{\KK^*}(\calA^*, \calS^*) < \infty$ by Proposition~\ref{rdbas.prop.big}.
  Now take any finite $\calA \in \KK$. By Theorem~\ref{sbrd.thm.big1} we have that
  $$
    T^{\mor}_{\KK}(\calA, \calS) \le \sum_{\calA^* \in U^{-1}(\calA)} T^{\mor}_{\KK^*}(\calA^*, \calS^*).
  $$
  Since both $L$ and $\calA$ are finite, it follows that $U^{-1}(\calA)$ is finite, the sum on the right is finite.
  Therefore, $T^{\mor}_{\KK}(\calA, \calS) < \infty$. Another application of Proposition~\ref{rdbas.prop.big}
  yields that $T_{\KK}(\calA, \calS) < \infty$.
\end{proof}

The fact that $(\QQ, \Boxed<)$ has finite big Ramsey degrees was established by Devlin in~\cite{devlin}
and the list of essential reducts of $(\QQ, \Boxed<)$ follows from a result
of Cameron presented in~\cite[Section 3.4]{cameron-oligo}. The five essential reducts of $(\QQ, \Boxed<)$ are
$(\QQ, \Boxed<)$ itself, the trivial structure $(\QQ, \0)$ and the three structures
$(\QQ, \mathrm{Betw})$, $(\QQ, \mathrm{Cyc})$ and $(\QQ, \mathrm{Sep})$ where:
\begin{align*}
  \mathrm{Betw}(x, y, z) &= x < y < z \lor z < y < x,\\
  \mathrm{Cyc}(x, y, z) &= x < y < z \lor y < z < x \lor z < x < y, \text{ and}\\
  \mathrm{Sep}(x, y, u, v) &= (\mathrm{Cyc}(x, y, u) \land \mathrm{Cyc}(x, v, y)) \lor (\mathrm{Cyc}(x, u, y) \land \mathrm{Cyc}(x, y, v)).
\end{align*}
Since all the essential reducts of $(\QQ, \Boxed<)$ are defined in $(\QQ, \Boxed<)$
by quantifier-free formulas, Theorem~\ref{rdbas.thm.reducts} applies and we have:

\begin{COR}
  All of the 5 essential reducts of $(\QQ, \Boxed<)$ have finite big Ramsey degrees.
\end{COR}
\begin{proof}
  Let us only show that $(\QQ, \mathrm{Betw})$ has finite big Ramsey degrees.
  Let $\KK^*$ be the class of all the finite and countably infinite chains, and let $\KK$ be the
  class of all the structures which are defined by $\Phi = \{\mathrm{Betw}\}$ in $\KK^*$. Then
  $(\QQ, \Boxed<)$ is universal for $\KK^*$ and $(\QQ, \mathrm{Betw})$ is defined in $(\QQ, \Boxed<)$ by $\Phi$.
  Since $(\QQ, \Boxed<)$ has finite big Ramsey degrees \cite{devlin}, so does $(\QQ, \mathrm{Betw})$.
\end{proof}

Let $\calR = (R, E^\calR)$ be the \emph{random graph}, the unique (up to isomorphism) undirected
countable ultrahomogeneous graph which is universal for the class of all the finite and countably infinite undirected graphs.
The fact that $\calR$ has finite big Ramsey degrees was established by Sauer in~\cite{Sauer-2006}
and the list of its essential reducts is due to Thomas~\cite{Thomas-1991}.
The five essential reducts of $\calR$ are $\calR$ itself, the trivial structure $(R, \0)$ and the three structures
$(R, \rho_3)$, $(R, \rho_4)$ and $(R, \rho_5)$ where $\rho_n \subseteq R^n$ is an $n$-ary relation on $R$ defined by
\begin{align*}
  (v_1, \ldots, v_n) \in \rho_n \text{ iff }
  & \text{the number of undirected edges in the}\\
  & \text{subgraph of $\calR$ induced by $v_1, \ldots, v_n$ is odd.}
\end{align*}
It is easy to see that each of the essential reducts of $\calR$ is defined in $\calR$
by a quantifier-free formula. So Theorem~\ref{rdbas.thm.reducts} applies and we have:

\begin{COR}
  All of the 5 essential reducts of $\calR$ have finite big Ramsey degrees.
\end{COR}

Let $\calT = (T, \Boxed\to)$ be the \emph{random tournament}, the unique (up to isomorphism)
countable ultrahomogeneous tournament which is universal for the class of all the finite and countably infinite
tournaments. The fact that $\calT$ has finite big Ramsey degrees was established by Sauer in~\cite{Sauer-2006}
and the list of its essential reducts is due to Bennett~\cite{Bennet-1997}.
The five essential reducts of $\calT$ are $\calT$ itself, the trivial structure $(T, \0)$ and the three structures
$(T, \mathrm{Betw'})$, $(T, \mathrm{Cyc'})$ and $(T, \mathrm{Sep'})$ defined as follows.
Let $\mathrm{Sep'}(x, y, u, v)$ be the first-order formula which expresses the fact that
$|\Boxed\to \sec (\{x, y\} \times \{u, v\})|$ is even, and let
\begin{align*}
  \mathrm{Betw'}(x, y, z) &= C(x, y, z) \lor C(z, y, x), \text{ and}\\
  \mathrm{Cyc'}(x, y, z) &= C(x, y, z) \lor D(x, z, y) \lor D(y, x, z) \lor D(z, y, x),
\end{align*}
where
\begin{align*}
  C(x, y, z) &= x \to y \land y \to z \land z \to x, \text{ and}\\
  D(x, y, z) &= x \to y \land y \to z \land x \to z.
\end{align*}
Since all the essential reducts of $\calT$ are defined in $\calT$ by quantifier-free formulas,
Theorem~\ref{rdbas.thm.reducts} applies and we have:

\begin{COR}
  All of the 5 essential reducts of $\calT$ have finite big Ramsey degrees.
\end{COR}

Quite recently Hubi\v cka proved in~\cite{hubicka-param-spaces} that the \emph{random poset},
the unique (up to isomorphism) countable ultrahomogeneous partially ordered set
which is universal for the class of all the finite and countably infinite partially ordered sets,
has finite big Ramsey degrees. Moreover, Hubi\v cka proves that free superpositions of finitely many
structures that have particular interpretations in the random poset (see~\cite{hubicka-param-spaces} for details)
also have finite big Ramsey degrees. In particular, it follows that the \emph{random ordered graph},
the unique (up to isomorphism) countable ultrahomogeneous ordered graph which is universal for the class of all
the finite and countably infinite ordered graphs (an ordered graph is a simple graph together
with a linear ordering of its vertices). Since the random graph, the random tournament and $(\QQ, \Boxed<)$
are all quantifier-free reducts of the random ordered graph (see~\cite{bpp} by
Bodirsky, Pinsker and Pongr\'acz for details and for the complete
list of first-order reducts of the random ordered graph), the results presented above follow immediately
from~\cite{hubicka-param-spaces} and~\cite{bpp}. We have nevertheless decided to keep the exposition so as to reflect
the historical development of the problem.

Going back to the random poset $\calP$ and Hubi\v cka's result~\cite{hubicka-param-spaces},
let us recall that the list of the essential reducts of $\calP$ were described in~\cite{ppppsz-2014}.
The five essential reducts of $\calP = (P, \Boxed{\le})$ are $\calP$ itself, the trivial structure $(P, \Boxed=)$
and the three structures $(P, \bot)$, $(T, \mathrm{Cyc}_\calP)$ and $(P, \mathrm{Par})$ defined as follows
(where $x < y$ stands for $x \le y \land x \ne y$):
\begin{align*}
  x \bot y = \mathstrut &x \not\le y \land y \not\le x,\\
  \mathrm{Cyc}_\calP(x, y, z) = \mathstrut&(x < y < z) \lor (y < z < x) \lor (z < x < y) \lor \\
                                &\lor (x < y \land x \bot z \land y \bot z)\\
                                &\lor (y < z \land y \bot x \land z \bot x)\\
                                &\lor (z < x \land z \bot y \land x \bot y), \text{and}\\
  \mathrm{Par}(x, y, z) = \mathstrut&x,\,y,\,z \text{ are distinct and the number of 2-element}\\
                          &\text{subsets of incomparable elements of } \{x,y,z\} \text{ is odd}.
\end{align*}

Since all the essential reducts of $\calP$ can be defined in $\calP$ by quantifier-free formulas,
Theorem~\ref{rdbas.thm.reducts} applies and we have:

\begin{COR}
  All of the 5 essential reducts of $\calP$ have finite big Ramsey degrees.
\end{COR}

Finally, we shall prove that $(\QQ, \Boxed<, 0)$ and all of its 116 essential reducts have finite big Ramsey
degrees. Since $(\QQ, \Boxed<, 0)$ is just $(\QQ, \Boxed<)$ with an additional constant, we shall
start by showing that adding constants to countable ultrahomogeneous structures
preserves the property of having finite big Ramsey degrees.

\begin{THM}\label{rdbas.thm.ccc}
  Let $L$ be a relational language, let $c_1, \ldots, c_n \notin L$ be new constant symbols and let
  $L' = L \union \{c_1, \ldots, c_n\}$.
  Let $\calU = (U, L^\calU)$ be a countably infinite ultrahomogeneous $L$-structure and let
  $\calU' = (U, L^\calU, u_1, \ldots, u_n)$ be an $L'$-structure obtained from $\calU$ by adding
  $n$ constants to the language. If $\calU$ has finite big Ramsey degrees then so does $\calU'$.
\end{THM}
\begin{proof}
  Let $\CC$ be the class of all the finite and countably infinite structures that embed into
  $\calU = (U, L^\calU)$
  and let $\DD$ be the class of all the finite and countably infinite structures that embed into
  $\calU' = (U, L^\calU, u_1, \ldots, u_n)$.
  We treat $\CC$ and $\DD$ as categories of structures by taking embeddings as morphisms.
  Assume that $\calU$ has finite big Ramsey degrees. The main idea of the proof is to use
  Theorem~\ref{bigrd.thm.1} to transport the property of having finite big Ramsey degrees from
  $\CC$ to $\DD$. Although $\DD$ is not a subcategory of $\CC$, it is easy to find a subcategory $\BB$
  of $\CC$ which is isomorphic to $\DD$ as follows.

  For a structure $\calA = (A, L^\calA, a_1, \ldots, a_n) \in \Ob(\DD)$ let $G(\calA) \in \Ob(\CC)$ be the
  $L$-structure which simply encodes the constants into the names of the elements of the structure as follows:
  $$
    G(\calA) = (A \times \{(a_1, \ldots, a_n)\}, L^{G(\calA)})
  $$
  where for each $R \in L$ we have that
  $$
    R^{G(\calA)} = \big\{ \big((x_1, a_1, \ldots, a_n), \ldots, (x_h, a_1, \ldots, a_n)\big) : (x_1, \ldots, x_h) \in R^\calA \big\}.
  $$
  This simple trick ensures that $G$ is injective on objects. Let us apply the same trick to morphisms.
  For $\calA = (A, L^\calA, a_1, \ldots, a_n)$, $\calB = (B, L^\calB, b_1, \ldots, b_n) \in \Ob(\DD)$
  and an embedding $f : \calA \hookrightarrow \calB$ define $G(f) : G(\calA) \hookrightarrow G(\calB)$ by
  $$
    G(f)(x, a_1, \ldots, a_n) = (f(x), b_1, \ldots, b_n).
  $$
  Then $G : \DD \to \CC$ is clearly a functor injective on both objects and hom-sets. Let $\BB$ be the
  subcategory of $\CC$ whose objects are of the form $G(\calA)$ for some $\calA \in \Ob(\DD)$ and nothing else,
  and whose morphisms are of the form $G(f)$ for some morphism $f$ in $\DD$ and nothing else.
  Then $\BB$ is a (not necessarily full) subcategory of $\CC$ isomorphic to $\DD$, so in order to
  complete the proof it suffices to show that $G(\calU') = \overline \calU =
  (\overline U, L^{\overline \calU})$ has finite big Ramsey degrees in~$\BB$.
  
  Take any $\calA' = (A, L^\calA, a_1, \ldots, a_n) \in \Ob(\DD)$ and let
  $G(\calA') = \overline \calA = (\overline A, L^{\overline \calA}) \in \Ob(\BB)$. Let
  $F : \Delta \to \BB$ be an $(\overline \calA, \overline \calU)$-diagram.
  Let $\Delta = T \union B$ where $T$ is the top row of~$\Delta$ and~$B$ is the
  bottom row of~$\Delta$, and let $(e_i : \overline \calU \to \calU)_{i \in T}$ be a commuting cocone over~$F$
  in~$\CC$:
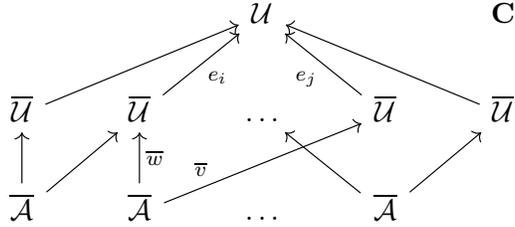
  \begin{center}
    \begin{tikzcd}
    & & \calU & & \CC
  \\
    \overline \calU \arrow[urr] & \overline \calU \arrow[ur, "e_i"'] & \dots & \overline \calU \arrow[ul, "e_j"] & \overline \calU \arrow[ull]
  \\
    \overline \calA \arrow[u] \arrow[ur] & \overline \calA \arrow[urr, "\overline v" near start] \arrow[u, "\overline w"'] & \dots & \overline \calA \arrow[ur] \arrow[ul] &
    \end{tikzcd}
  \end{center}

  To prove that the diagram $F$ has a commuting cocone in $\BB$ we have to construct
  an object $\overline \calV \in \Ob(\BB)$ and morphisms $\overline f_i : \overline \calU \to \overline \calV$, $i \in T$,
  so that the diagram analogous to the above one commutes. The idea we are going to implement is straightforward:
  we shall start with a substructure $\calV = (V, L^\calV)$ of $\calU$ induced by $V = \bigcup_{i \in T} e_i(\overline U)$.
  We shall then identify some convenient $v_1, \ldots, v_n \in V$, prove that
  $\calV' = (V, L^\calV, v_1, \ldots, v_n) \in \Ob(\DD)$ and put $\overline \calV = G(\calV')$
  at the tip of the commuting cocone in $\BB$. The morphisms
  $\overline f_i : \overline \calU \to \overline \calV$ will be appropriate modifications of
  the codomain restrictions of $e_i$, $i \in T$. The trickiest part in the entire construction is the identification
  of $v_1, \ldots, v_n \in V$ that can act as constants in $\calV'$. Since the cocone morphisms $\overline f_i$
  are going to be the codomain restrictions of $e_i$ (modulo renaming of elements), in order to identify the elements of $U$ that
  can act as constants in $\calV'$ we have to ensure that
  $$
    e_i(u_m, u_1, \ldots, u_n) = e_j(u_m, u_1, \ldots, u_n),
  $$
  for all $i, j \in T$ and $1 \le m \le n$. (Recall that $(u_m, u_1, \ldots, u_n)$, $1 \le m \le n$, are the constants
  of $\calU'$ in disguise.) Once this is ensured we will take
  $$
    v_m = e_{t_0}(u_m, u_1, \ldots, u_n), \quad 1 \le m \le n,
  $$
  for an arbitrary but fixed $t_0 \in T$.
  
  In order to carry out this program we need the notion of the connected component of a binary digraph
  (see the discussion preceding Theorem~\ref{bigrd.thm.1}).
  A \emph{walk} between two elements $x$ and $y$ of the top row of a binary digraph 
  consists of some vertices $x = t_0$, $t_1$, \dots, $t_k = y$ of the top row, some vertices
  $s_1$, \dots, $s_k$ of the bottom row, and arrows $s_{j} \to t_{j-1}$ and $s_{j} \to t_{j}$, $1 \le j \le k$:
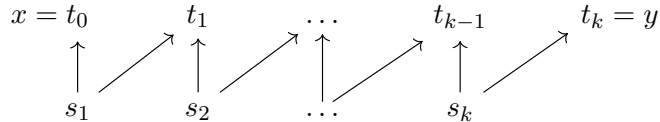
  \begin{center}
    \begin{tikzcd}
      \llap{$x = \mathstrut$}t_0 & t_1 & \dots & t_{k-1} & t_k\rlap{$\mathstrut = y$} \\
      s_1 \arrow[u] \arrow[ur] & s_2 \arrow[u] \arrow[ur] & \dots \arrow[u] \arrow[ur] & s_k \arrow[u] \arrow[ur]
    \end{tikzcd}
  \end{center}
  A binary digraph is \emph{connected} if there is a walk between any pair of distinct vertices of the top row.
  A \emph{connected component} of a binary digraph $\Delta$ is a maximal (with respect to inclusion) set $S$ of vertices
  of the top row such that there is a walk between any pair of distinct vertices from~$S$.
  (Note that $s_j$'s are not required to be distinct.)

  Let $S \subseteq T$ be a connected component of $\Delta$ and let us show that
  $$
    e_i(u_m, u_1, \ldots, u_n) = e_j(u_m, u_1, \ldots, u_n),
  $$
  for all $i, j \in S$ and $1 \le m \le n$.
  Take any $i, j \in S$. Since~$S$ is a connected component of~$\Delta$,
  there exist $i = t_0$, $t_1$, \ldots, $t_k = j$ in $S$,
  $s_1$, \ldots, $s_k$ in $B$ and arrows $p_j : s_{j} \to t_{j-1}$ and $q_j : s_{j} \to t_{j}$, $1 \le j \le k$:
  \begin{center}
    \begin{tikzcd}
      \llap{$i = \mathstrut$}t_0 & t_1 & \ldots & t_{k-1} & t_k\rlap{$\mathstrut = j$} \\
      s_1 \arrow[u, "p_1"] \arrow[ur, "q_1"'] & s_2 \arrow[u] \arrow[ur] & \ldots \arrow[u] \arrow[ur] & s_k \arrow[u, "p_k"] \arrow[ur, "q_k"']
    \end{tikzcd}
  \end{center}
  Let $F(p_j) = \overline w_j$ and $F(q_j) = \overline v_j$, $1 \le j \le k$. Then
  \begin{align*}
    e_i&(u_m, u_1, \ldots, u_n) =\\
      &= e_{t_0}(u_m, u_1, \ldots, u_n) && [i = t_0] \\
      &= e_{t_0}(\overline w_1(a_m, a_1, \ldots, a_n)) && [\overline w_1(a_m, a_1, \ldots, a_n) = (u_m, u_1, \ldots, u_n)]\\
      &= e_{t_1}(\overline v_1(a_m, a_1, \ldots, a_n)) && \\
      &= e_{t_1}(u_m, u_1, \ldots, u_n) && [\overline v_1(a_m, a_1, \ldots, a_n) = (u_m, u_1, \ldots, u_n)].
  \end{align*}
  because $(e_i)_{i \in T}$ is a commuting cocone over $F$ whence
  \begin{center}
    \begin{tikzcd}
      & \calU & 
    \\
       \overline \calU \arrow[ur, "e_{t_0}"] &  & \overline \calU \arrow[ul, "e_{t_1}"'] 
    \\
       \overline \calA \arrow[urr, "\overline v_1 = F(q_1)"' near start] \arrow[u, "F(p_1) = \overline w_1"] 
    \end{tikzcd}
  \end{center}
  (recall that $\overline v_1$ and $\overline w_1$ are morphisms in $\BB$).
  Analogously, $e_{t_1}(u_m, u_1, \ldots, u_n) = e_{t_2}(u_m, u_1, \ldots, u_n)$
  and so on. Thus,
  \begin{align*}
    e_i(u_m, u_1, \ldots, u_n) &= e_{t_0}(u_m, u_1, \ldots, u_n) = \dots \\
    \ldots &= e_{t_k}(u_m, u_1, \ldots, u_n) = e_j(u_m, u_1, \ldots, u_n).
  \end{align*}

  In contrast to that, if $S, S' \subseteq T$ are two distinct connected components of $\Delta$ we cannot guarantee that
  $e_i(u_m, u_1, \ldots, u_n) = e_j(u_m, u_1, \ldots, u_n)$ for $i \in S$ and $j \in S'$. We shall now modify the commuting
  cocone $(e_i : \overline \calU \to \calU)_{i \in T}$ so as to ensure that this is always the case.

  Let $\{S_\alpha : \alpha < \lambda\}$ be the set of all the connected
  components of $\Delta$, where $S_\alpha \subseteq T$, $\alpha < \lambda$.
  Take any ordinal $\alpha$ such that $0 < \alpha < \lambda$. Let $i \in S_0$ and $j \in S_\alpha$
  be arbitrary and let $p : s \to i$ and $q : s' \to j$ be two arrows, one in the part of $\Delta$
  determined by~$S_0$ and the other one in the part of $\Delta$ determined by~$S_\alpha$.
  Let $\overline w = F(p)$ and $\overline v = F(q)$:
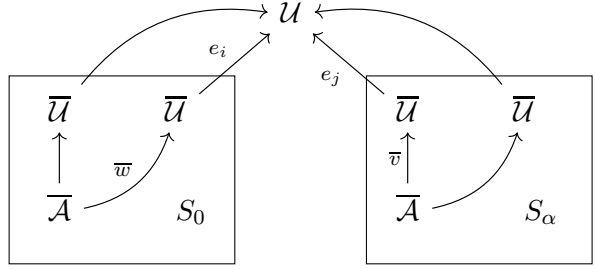
  \begin{center}
    \begin{tikzcd}[execute at end picture={
            \draw (1,-2) rectangle (4,0.5);
            \draw (-3.75,-2) rectangle (-0.75,0.5);
        }]
        & & \calU & &
      \\
        \overline \calU \arrow[urr, bend left] & \overline \calU \arrow[ur, "e_i"] &  & \overline \calU \ar[ul, "e_j"] & \overline \calU \arrow[ull, bend right]
      \\
        \overline \calA \arrow[u] \arrow[ur, bend right, "\overline w"]  & \rlap{$S_0$} & & \overline \calA \arrow[ur, bend right] \arrow[u, "\overline v"] & \rlap{$S_\alpha$}
    \end{tikzcd}
  \end{center}
  Let $\overline A_0 = \{(a_m, a_1, \ldots, a_n) : 1 \le m \le n\}$ and let
  $\overline \calA_0 = \overline \calA[\,\overline A_0]$ be the substructure of $\overline \calA$ induced by $\overline A_0$.
  Then $\restr{e_i \circ \overline w}{\overline A_0} : \overline \calA_0 \hookrightarrow \calU$ and
  $\restr{e_j \circ \overline v}{\overline A_0} : \overline \calA_0 \hookrightarrow \calU$ are two distinct embeddings
  of the same finite structure $\overline \calA_0$ into $\calU$.
  Since $\calU$ is ultrahomogeneous there is an $h_\alpha \in \Aut(\calU)$ such that
  $\restr{e_i \circ \overline w}{\overline A_0} = \restr{h_\alpha \circ e_j \circ \overline v}{\overline A_0}$.
  Put $h_0 = \id_{\calU}$ and let $\alpha(i)$ be the unique ordinal such that $i \in S_{\alpha(i)}$.
  Analogously, let $\overline \calU_0 = \overline \calU[\, \overline U_0]$ where
  $$
    \overline U_0 = \{(u_m, u_1, \ldots, u_n) : 1 \le m \le n\}.
  $$
  Then $\restr{h_{\alpha(i)} \circ e_i}{\overline U_0} = \restr{h_{\alpha(j)} \circ e_j}{\overline U_0}$ for $i, j \in T$
  (this follows from the fact that $\overline w(\overline A_0) = \overline U_0 = \overline v(\overline A_0)$ and
  $\restr{\overline w}{\overline A_0} = \restr{\overline v}{\overline A_0}$),
  so $(h_{\alpha(i)} \circ e_i : \overline \calU \to \calU)_{i \in T}$ is
  still a commuting cocone over~$F$ in $\CC$:
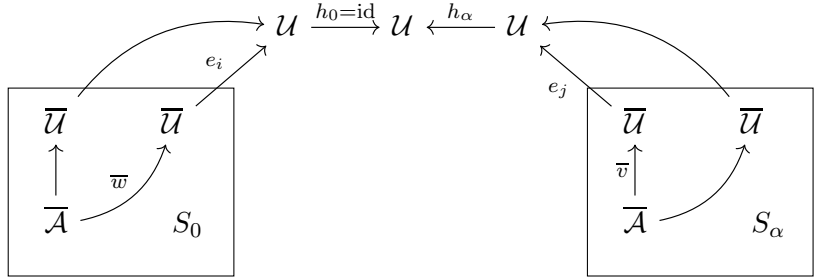
  \begin{center}
    \begin{tikzcd}[execute at end picture={
            \draw (2.45,-2) rectangle (5.45,0.5);
            \draw (-5.25,-2) rectangle (-2.25,0.5);
        }]
    & & \calU \arrow[r, "h_0 = \id"] & \calU & \calU \ar[l, "h_\alpha"'] & & 
  \\
    \overline \calU \arrow[urr, bend left] & \overline \calU \arrow[ur, "e_i"] & & &  & \overline \calU \arrow[ul, "e_j"] & \overline \calU \arrow[ull, bend right]
  \\
    \overline \calA \arrow[u] \arrow[ur, bend right, "\overline w"]  & \rlap{$S_0$} & & & & \overline \calA \arrow[ur, bend right] \arrow[u, "\overline v"] & \rlap{$S_\alpha$}
    \end{tikzcd}
  \end{center}
  and for this commuting cocone we have that
  $$
    h_{\alpha(i)} \circ e_i(u_m, u_1, \ldots, u_n) = h_{\alpha(j)} \circ e_j(u_m, u_1, \ldots, u_n),
  $$
  for all $i, j \in T$ and $1 \le m \le n$.

  Hence, without loss of generality we can assume that $(e_i : \overline \calU \to \calU)_{i \in T}$
  is a commuting cocone over~$F$ in $\CC$ such that
  \begin{equation}\label{rdbas.eq.e}
    e_i(u_m, u_1, \ldots, u_n) = e_j(u_m, u_1, \ldots, u_n)
  \end{equation}
  for all $i, j \in T$ and $1 \le m \le n$.
  Let $V = \bigcup_{i \in T} e_i(\overline U)$ and let $\calV = \calU[V]$ be the substructure of $\calU$ induced by~$V$.
  Take an arbitrary but fixed $t_0 \in T$ and put
  \begin{equation}\label{rdbas.eq.v}
    v_m = e_{t_0}(u_m, u_1, \ldots, u_n) \in V, \quad 1 \le m \le n.
  \end{equation}
  Let $\calV' = (V, L^\calV, v_1, \ldots, v_n)$. To show that $\calV' \in \Ob(\DD)$ we have to show that
  $\calV'$ embeds into $\calU'$. Recall, first, that $\overline \calU \cong \calU$ and that
  $\overline \calU[\,\overline U_0]$ is isomorphic to $\calU[u_1, \ldots, u_n]$
  where the isomorphism is $\phi : \overline U_0 \to \{u_1, \ldots, u_n\}$
  given by
  $$
    \phi(u_m, u_1, \ldots, u_n) = u_m, \quad 1 \le m \le n.
  $$
  On the other hand, $\overline \calU[\,\overline U_0]$ is isomorphic to $\calV[v_1, \ldots, v_n]$
  where the isomorphism is $\restr{e_{t_0}}{\overline U_0}$. Therefore,
  $\calU[u_1, \ldots, u_n]$ and $\calV[v_1, \ldots, v_n]$ are isomorphic and the isomorphism is
  $\psi : \{v_1, \ldots, v_n\} \to \{u_1, \ldots, u_n\} : v_i \mapsto u_i$, $1 \le i \le n$. Since $\calU$ is
  ultrahomogeneous there is a $\hat\psi \in \Aut(\calU)$ which extends~$\psi$, so
  $\restr{\hat\psi}{V}$ is an embedding of $\calV$ into $\calU$ which takes $v_i$ to $u_i$, $1 \le i \le n$.
  In other words, $\restr{\hat\psi}{V} : \calV' \hookrightarrow \calU'$ whence $\calV' \in \Ob(\DD)$.

  Let us now construct a commuting cocone over $F$ in $\BB$. Let $\overline \calV = G(\calV') \in \Ob(\BB)$.
  To define the morphisms $\overline \calU \to \overline \calV$ consider, first, the mappings
  $f_i : \calU' \to \calV'$ in $\DD$, $i \in T$, defined by:
  $$
    f_i(x) = e_i(x, u_1, \ldots, u_n).
  $$
  Each $f_i$ is an embedding of $\calU$ into $\calV$ such that for $1 \le m \le n$:
  \begin{align*}
    f_i(u_m) &= e_i(u_m, u_1, \ldots, u_n)\\
             &= e_{t_0}(u_m, u_1, \ldots, u_n) && \text{by \eqref{rdbas.eq.e}}\\
             &= v_m. && \text{by \eqref{rdbas.eq.v}}
  \end{align*}
  Hence $f_i : \calU' \hookrightarrow \calV'$, $i \in T$, is a morphism in $\DD$.
  Finally, for each $i \in T$ put $\overline f_i = G(f_i) : \overline \calU \to \overline \calV$
  and let us show that$(\overline f_i : \overline \calU \to \overline \calV)_{i \in T}$ is a commuting cocone over~$F$ in $\BB$.
  Assume that in the original cocone over $F$ we have that $e_i \circ \overline w = e_j \circ \overline v$
  where $\overline w = G(w)$ for some $w : \calA' \hookrightarrow \calU'$ and
  $\overline v = G(v)$ for some $v : \calA' \hookrightarrow \calU'$:
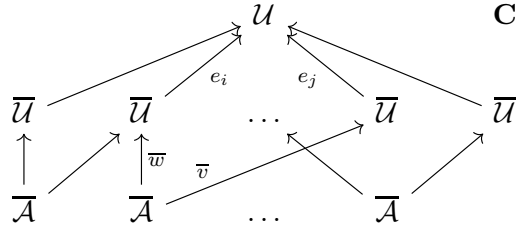
  \begin{center}
    \begin{tikzcd}
    & & \calU & & \CC
  \\
    \overline \calU \arrow[urr] & \overline \calU \arrow[ur, "e_i"'] & \dots & \overline \calU \arrow[ul, "e_j"] & \overline \calU \arrow[ull]
  \\
    \overline \calA \arrow[u] \arrow[ur] & \overline \calA \arrow[urr, "\overline v" near start] \arrow[u, "\overline w"'] & \dots & \overline \calA \arrow[ur] \arrow[ul] &
    \end{tikzcd}
  \end{center}
  Then
  \begin{align*}
    \overline f_i \circ \overline w (x, a_1, \ldots, a_n)
      &= \overline f_i (w(x), u_1, \ldots, u_n)\\
      &= (f_i(w(x)), v_1, \ldots, v_n)\\
      &= (e_i(w(x), u_1, \ldots, u_n), v_1, \ldots, v_n)\\
      &= (e_i \circ \overline w (x, a_1, \ldots, a_n), v_1, \ldots, v_n)\\
      &= (e_j \circ \overline v (x, a_1, \ldots, a_n), v_1, \ldots, v_n)\\
      &= (e_j(v(x), u_1, \ldots, u_n), v_1, \ldots, v_n)\\
      &= (f_j(v(x)), v_1, \ldots, v_n)\\
      &= \overline f_j (v(x), u_1, \ldots, u_n)\\
      &= \overline f_j \circ \overline v (x, a_1, \ldots, a_n).
  \end{align*}
  This completes the proof.
\end{proof}

\begin{COR}
  $(\QQ, \Boxed<, 0)$ has finite big Ramsey degrees.
\end{COR}
\begin{proof}
  Immediate from the fact that $(\QQ, \Boxed<)$ has finite big Ramsey degrees~\cite{devlin} and
  Theorem~\ref{rdbas.thm.ccc}.
\end{proof}

All the essential reducts of $(\QQ, \Boxed<, 0)$, and much more,
were classified by Junker and Ziegler in~\cite{junker-ziegler}.
It turns out that there are 116 of them and that they are all defined by quantifier-free formulas in $(\QQ, \Boxed<, 0)$.
So Theorem~\ref{rdbas.thm.reducts} applies and we have:

\begin{COR}
  All of the 116 essential reducts of $(\QQ, \Boxed<, 0)$ have finite big Ramsey degrees.
\end{COR}

\section{Acknowledgements}

The author would like to thank two anonymous referees for thorough reading of the manuscript and
for many helpful suggestions on how to correct the inaccuracies in the previous version of the paper.

The author gratefully acknowledges the financial support of the Ministry of Education, Science and Technological Development
of the Republic of Serbia (Grant No.\ 451-03-68/2020-14/200125).


\begin{thebibliography}{99}\frenchspacing
\bibitem{AHS}
  J. Ad\'amek, H. Herrlich, G. E. Strecker.
  Abstract and Concrete Categories: The Joy of Cats.
  Dover Books on Mathematics, Dover Publications 2009
\bibitem{Bennet-1997}
  J. H. Bennett.
  The reducts of some infinite homogeneous graphs and tournaments. PhD thesis, Rutgers university, 1997.
\bibitem{bpp}
  M. Bodirsky, M. Pinsker, A. Pongr\'acz. The 42 reducts of the random ordered graph.
  Proc. LMS 113:3 (2015) 591--632.
\bibitem{cameron-oligo}
  P. J. Cameron.
  Oligomorphic permutation groups.
  Cambridge University Press, Cambridge, 1990.
\bibitem{dasilvabarbosa}
  K. Dasilva Barbosa.
  A Categorical Notion of Precompact Expansions.
  (Preprint available as arXiv:2002.11751)
\bibitem{devlin}
  D. Devlin.
  Some partition theorems and ultrafilters on $\omega$.
  Ph.D. Thesis, Dartmouth College, 1979.
\bibitem{fouche97}
  W. L. Fouch\'e.
  Symmetry and the Ramsey degree of posets.
  Discrete Math. 167/168 (1997), 309--315.
\bibitem{fouche98}
  W. L. Fouch\'e.
  Symmetries in Ramsey theory.
  East–West J. Math. 1 (1998), 43--60.
\bibitem{fouche99}
  W. L. Fouch\'e.
  Symmetry and the Ramsey degrees of finite relational structures.
  J. Comb. Theory Ser. A 85 (1999), 135--147.
\bibitem{galvin1}
  F. Galvin.
  Partition theorems for the real line.
  Notices Amer. Math. Soc. 15 (1968), 660.
\bibitem{galvin2}
  F. Galvin.
  Errata to ``Partition theorems for the real line''.
  Notices Amer. Math. Soc. 16 (1969), 1095.
\bibitem{GLR}
  R. L. Graham, K. Leeb, B. L. Rothschild.
  Ramsey's theorem for a class of categories.
  Advances in Math.\ 8 (1972), 417--443; errata 10 (1973), 326--327
\bibitem{hubicka-param-spaces}
  J. Hubi\v cka.
  Big Ramsey degrees using parameter spaces.
  Preprint, arXiv:2009.00967.
\bibitem{junker-ziegler}
  M. Junker, M. Ziegler.
  The 116 reducts of $(Q, \Boxed<, a)$.
  Journal of Symbolic Logic, 74(2008), 861--884.
\bibitem{KPT}
  A. S. Kechris, V. G. Pestov, S. Todor\v cevi\'c.
  Fra\"\i ss\'e limits, Ramsey theory and topological dynamics of automorphism groups.
  Geom.\ Funct.\ Anal.\ 15 (2005), 106--189
\bibitem{leeb-cat}
    K.\ Leeb.
    The categories of combinatorics.
    Combinatorial structures and their applications. Gordon and Breach, New York (1970).
\bibitem{masul-bigrd}
  D. Ma\v sulovi\'c.
  Finite big Ramsey degrees in universal structures.
  Journal of Combinatorial Theory Ser. A 170 (2020), 105--137
\bibitem{masul-dual-kpt}
  D. Ma\v sulovi\'c.
  The Kechris-Pestov-Todor\v cevi\' c correspondence from the point of view of category theory.
  Applied Categorical Structures 29 (2021), 141--169
\bibitem{masulovic-ramsey}
  D.\ Ma\v sulovi\'c, L.\ Scow.
  Categorical equivalence and the Ramsey property for finite powers of a primal algebra.
  Algebra Universalis 78 (2017), 159--179
\bibitem{masul-sobot}
  D. Ma\v sulovi\'c, B.\ \v Sobot.
  Countable ordinals and big Ramsey degrees.
  (to appear in Combinatorica)
\bibitem{mu-pon}
  M. M\"uller, A. Pongr\'acz.
  Topological dynamics of unordered Ramsey structures.
  Fund.\ Math.\ 230 (2015), 77--98
\bibitem{N1995}
  J.\ Ne\v set\v ril.
  Ramsey theory. In: R.\ L.\ Graham, M.\ Gr\"otschel and L.\ Lov\'asz, eds, Handbook of Combinatorics, Vol.~2,
  1331--1403, MIT Press, Cambridge, MA, USA, 1995.
\bibitem{Nesetril}
  J. Ne\v set\v ril.
  Ramsey classes and homogeneous structures.
  Combin.\ Probab.\ Comput.\ 14 (2005), 171--189
\bibitem{Nesetril-Rodl}
  J.\ Ne\v set\v ril, V.\ R\"odl.
  Partitions of finite relational and set systems.
  J.\ Combin.\ Theory Ser.\ A 22 (1977), 289--312.
\bibitem{vanthe-ufcs}
  L. Nguyen Van Th\'e.
  Universal flows of closed subgroups of $S_\infty$ and relative extreme amenability.
  Asymptotic Geometric Analysis, Fields Institute Communications 68 (2013), 229--245
\bibitem{vanthe-more}
  L.\ Nguyen Van Th\'e.
  More on the Kechris-Pestov-Todorcevic correspondence: precompact expansions.
  Fund.\ Math.\ 222 (2013), 19--47
\bibitem{ppppsz-2014}
  P. P. Pach, M. Pinsker, G. Pluh\'ar, A. Pongr\'acz, Cs. Szab\'o.
  Reducts of the random partial order.
  Advances in Mathematics, 267 (2014), 94--120.
\bibitem{Sauer-2006}
  N. W. Sauer.
  Coloring subgraphs of the Rado graph.
  Combinatorica 26 (2006), 231--253.
\bibitem{P-V}
  H. J. Pr\"omel, B. Voigt.
  Hereditary attributes of surjections and parameter sets.
  European J.\ Combin.\ 7 (1986), 161--170
\bibitem{Ramsey}
  F.\ P.\ Ramsey.
  On a problem of formal logic.
  Proc.\ London Math.\ Soc.\ 30 (1930), 264--286.
\bibitem{Thomas-1991}
  S. Thomas.
  Reducts of the random graph.
  Journal of Symbolic Logic, 56 (1991), 176--181.
\bibitem{Zucker-1}
  A. Zucker.
  Topological dynamics of automorphism groups, ultrafilter combinatorics and the Generic Point Problem.
  Trans. Amer. Math. Soc. 368 (2016), 6715--6740.
\bibitem{Zucker-2}
  A. Zucker.
  Big Ramsey degrees and topological dynamics.
  Groups, Geom., Dyn., 13 (2019), 235--276
\end{thebibliography}
\end{document}